\documentclass{conm-p-l}
\usepackage[english]{babel}
\usepackage{enumerate}
\usepackage{amsmath,amssymb}
\usepackage{algorithm}
\usepackage{algpseudocode}
\algnewcommand\algorithmicinput{\textbf{Input:}}
\algnewcommand\INPUT{\item[\algorithmicinput]}
\algnewcommand\algorithmicoutput{\textbf{Output:}}
\algnewcommand\OUTPUT{\item[\algorithmicoutput]}
\algnewcommand\algorithmiccomplexity{\textbf{Complexity:}}
\algnewcommand\COMPLEXITY{\item[\algorithmiccomplexity]}
\algnewcommand\algorithmicproc{\textbf{Procedure:}}
\algnewcommand\PROCEDURE{\item[\algorithmicproc]}
\newlength{\continueindent}
\usepackage{etoolbox}
\makeatletter
\newcommand*{\ALG@customparshape}{\parshape 2 \leftmargin \linewidth \dimexpr\ALG@tlm+\continueindent\relax \dimexpr\linewidth+\leftmargin-\ALG@tlm-\continueindent\relax}
\apptocmd{\ALG@beginblock}{\ALG@customparshape}{}{\errmessage{failed to patch}}
\makeatother
\ifx\pdfpageheight\undefined
   \usepackage[dvips,colorlinks=true,linkcolor=blue,citecolor=red,%
      urlcolor=green]{hyperref}
   \usepackage[dvips]{graphicx}
   \makeatletter
   \edef\Gin@extensions{\Gin@extensions,.mps}
   \makeatother
\else
 \usepackage[pdftex]{graphicx}
   \usepackage[bookmarksopen=false,pdftex=true,breaklinks=true,%
      backref=page,pagebackref=true,plainpages=false,%
      hyperindex=true,pdfstartview=FitH,colorlinks=true,%
      pdfpagelabels=true,colorlinks=true,linkcolor=blue,%
      citecolor=red,urlcolor=green,hypertexnames=false%
      ]%
   {hyperref}
\fi
\usepackage{bm}
\usepackage{amsmath}
\usepackage{tabularx}
\usepackage{color, colortbl}
\usepackage{url}
\makeindex
\catcode`\<=\active \def<{
\fontencoding{T1}\selectfont\symbol{60}\fontencoding{\encodingdefault}}
\catcode`\>=\active \def>{
\fontencoding{T1}\selectfont\symbol{62}\fontencoding{\encodingdefault}}
\catcode`\|=\active \def|{
\fontencoding{T1}\selectfont\symbol{124}\fontencoding{\encodingdefault}}

\newcommand{\noplus}{}

\newcommand{\tmop}[1]{\ensuremath{\operatorname{#1}}}

\definecolor{LightCyan}{rgb}{0.88,1,1}

\newtheorem{theorem}{Theorem}
\newtheorem{lemma}{Lemma}
\newtheorem{corollary}{Corollary}
\newtheorem{proposition}{Proposition}

\theoremstyle{definition}
\newtheorem{definition}{Definition}
\newtheorem{example}{Example}

\newtheorem{notation}{Notation}

\theoremstyle{remark}
\newtheorem{remark}{Remark}

\newtheoremstyle{break}  
  {3pt}   
  {11pt}   
  {\normalfont}  
  {0pt}       
  {\scshape} 
  {}         
  {4pt}  
  {}          
\theoremstyle{break}

\newcommand{\ra}{\rangle}
\newcommand{\la}{\langle}

\newcommand{\F}{\mathbb{F}}
\newcommand{\R}{\mathrm{R}}
\newcommand{\C}{\mathrm{C}}
\newcommand{\Q}{\mathbb{Q}}

\newcommand{\ZZ}{\mathrm{Zer}}
\newcommand{\RR}{\mathrm{Reali}}
\newcommand{\D}{\mathrm{D}}
\newcommand{\Z}{\mathbb{Z}}
\newcommand{\eps}{\varepsilon}
\newcommand{\chiep}{\chi}
\DeclareMathOperator{\SIGN}{SIGN}
\DeclareMathOperator{\sign}{sign}
\DeclareMathOperator{\Ext}{Ext}
\DeclareMathOperator{\Def}{Def}
\DeclareMathOperator{\Der}{Der}
\DeclareMathOperator{\EQ}{EuQ}
\DeclareMathOperator{\lex}{lex}

\newcommand{\HH}{\mathrm{H}}
\newcommand{\card}{\mathrm{card}}

\newcommand {\hide}[1]{}

\newcommand{\X}{\mathbf{X}}
\newcommand{\x}{\mathbf{x}}

\newcommand{\N}{\mathbb{N}}
\newcommand{\ZB}{\mathbf{Z}}
\newcommand{\zb}{\mathbf{z}}
\newcommand{\kk}{\mathbf{k}}

\newcommand{\length}{\mathrm{length}}
\newcommand{\gen}{\mathrm{gen}}

\newcommand{\boldpi}{\boldsymbol{\pi}}
\newcommand{\fixed}{\mathrm{fixed}}
\newcommand{\orbit}{\mathrm{orbit}}
\newcommand{\ind}{\mathrm{ind}}

\newcommand{\Hess}{\mathrm{Hess}}

\newcommand{\EuQ}{\mathrm{EuQ}}

\newcommand{\boldPi}{\boldsymbol{\Pi}}

\begin{document}
\title[Computing the Euler-Poincar{\'e} characteristic ]
{Efficient algorithms for computing the 
Euler-Poincar{\'e} characteristic of symmetric semi-algebraic sets}

\author{Saugata Basu}
\address{Department of Mathematics\\
Purdue University, West Lafayette\\
USA}
\email{sbasu@math.purdue.edu
}

\author{Cordian Riener}
\address{Aalto Science Institute\\
Aalto University, Espoo\\
Finland}
\email{cordian.riener@aalto.fi}

\thanks{The first author was partially supported by NSF grants
CCF-1319080 and DMS-1161629.}

\maketitle

\begin{abstract}
Let $\R$ be a real closed field and $\D \subset \R$ an ordered domain.
We consider the algorithmic problem of computing the generalized Euler-Poincar\'e characteristic of real algebraic as well as semi-algebraic
subsets of $\R^k$, which are defined by symmetric polynomials with coefficients in $\D$. We give algorithms for computing the generalized Euler-Poincar\'e characteristic
of such sets, whose complexities measured by the number the number of arithmetic operations in $\D$, are polynomially bounded in terms of $k$ and the number of polynomials
in the input, assuming that the degrees of the input polynomials are bounded by a constant.
This is in contrast to the best complexity of the known algorithms for the same problems in 
the non-symmetric situation, which are singly exponential. This singly exponential complexity for the 
latter problem is unlikely to be improved because of hardness result ($\#\mathbf{P}$-hardness) coming from 
discrete complexity theory. 
\hide{
We give   algorithms with polynomially bounded complexities (for fixed degrees) for
computing the generalized Euler-Poincar{\'e} characteristic of semi-algebraic sets defined by symmetric 
polynomials. This is in contrast to the best complexity of the known algorithms for the same problem in 
the non-symmetric situation, which is singly exponential. This singly exponential complexity for the 
latter problem is unlikely to be improved because of hardness result ($\#\mathbf{P}$-hardness) coming from 
discrete complexity theory. 
}
\end{abstract}

\section{Introduction}

Let $\R$ be a real closed field which is fixed for the remainder of the paper,
and let $\C$ denote the algebraic closure of $\R$. Given a semi-algebraic set $S\subset \R^k$, i.e., a set defined by unions and intersections of polynomial inequalities, it is a fundamental question of 
computational algebraic geometry to compute topological information about $S$. 
This problem of designing efficient algorithms for computing topological
invariants -- such as the Betti numbers as well as the Euler-Poincar{\'e}
characteristic -- has a long history. The first algorithms {\cite{SS}} used
the technique of cylindrical algebraic decomposition and consequently had
doubly exponential complexity. Algorithms for computing the zeroth Betti
number (i.e. the number of semi-algebraically connected components) of
semi-algebraic sets using the critical points method were discovered later
{\cite{Canny87,GR92,GV92,BPR99}} and improving
this complexity bound remains an active area of research even today. Later,
algorithms with singly exponential complexity for computing the first Betti
number {\cite{BPRbettione}}, as well as the first few Betti numbers
{\cite{Bas05-first}} were discovered. Algorithms with singly exponential
complexity for computing the Euler-Poincar\'e characteristic are also known
{\cite{BPR-euler-poincare}}. It remains an open problem to design an algorithm
with singly exponential complexity for computing all the Betti numbers of a
given semi-algebraic set. Algorithms with polynomially bounded complexity for
computing the Betti numbers of semi-algebraic sets are known in a few cases
-- for example, for sets defined by a few (i.e. any constant number of)
quadratic inequalities \cite{BP'R07joa,BP'R07jems}. Also note that
the problem of expressing the Euler-Poincar\'e characteristic of real
algebraic varieties in terms of certain algebraic invariants of the
polynomials defining the variety has been considered by several other authors
(see for example {\cite{Dutertre2003}} and {\cite{Szafraniec94}}). But these
studies do not take into account the computational complexity aspect of the
problem.

In this article we will restrict to the case when $S$ is defined by \emph{symmetric} polynomial inequalities whose degree is at most $d\in\N$, which we will think of as a fixed constant. It is known \cite{Timofte03,Riener} that in this particular setup one can decide emptiness of $S$ in a time which is polynomial in $k$ - the number of variables.  Despite being a rather basic property, it is known to be a 
$\textbf{NP}$-hard problem (in the Blum-Shub-Smale model) to decide
if a given real algebraic variety $V \subset \R^{k}$ defined by one polynomial
equation of degree at most $4$ is empty or not {\cite{BSS89}}.
Following this notable difference   it is natural to ask, if in general symmetric semi-algebraic sets are
algorithmically more tractable than general semi-algebraic sets and if it is possible to obtain polynomial time (for fixed degree $d$)
algorithms for computing topological invariants of such sets. In this article  we answer this in
the affirmative for the problem of computing the generalized
Euler-Poincar{\'e} characteristic (both the ordinary as well as the
equivariant versions) of symmetric semi-algebraic sets. The problem of
computing the generalized Euler-Poincar{\'e} characteristic is important in
several applications both theoretical and practical. For example, such an
algorithm is a key ingredient in computing the integral (with respect to the
Euler-Poincar{\'e} measure) of constructible functions, and this latter
problem has been of recent interest in several applications
\cite{Ghrist2010}.

Before proceeding further we first fix some notation.

\begin{notation}
  For $P \in \R [X_{1}, \ldots,X_{k} ]$ (respectively $P \in \C [ X_{1},
  \ldots,X_{k} ]$) we denote by $\ZZ (P, \R^{k})$ (respectively, $\ZZ(P,\C^k)$) the set of zeros of $P$ in
  $\R^{k}$ (respectively, $\C^k$). More generally, for any finite set
  $\mathcal{P} \subset \R [ X_{1}, \ldots,X_{k} ]$ (respectively, $\mathcal{P} \subset \C [ X_{1}, \ldots,X_{k} ]$), we denote by $\ZZ(\mathcal{P},\C^k)$
  $\ZZ (\mathcal{P}, \R^{k})$ (respectively,  $\ZZ(\mathcal{P},\C^k)$)  the set of common zeros of $\mathcal{P}$ in
  $\R^{k}$ (respectively, $\C^k$). 
\end{notation}

\begin{notation}

Let $\mathcal{P}\subset \R [ X_{1}, \ldots,X_{k} ]$ be a finite family of polynomials. 
\begin{enumerate}
\item We call  any Boolean formula $\Phi$ with atoms, $P \; \sim \; 0, P \in \mathcal{P}$, where $\sim$ is one of $=,>,$ or $<$, 
  to be a \emph{$\mathcal{P}$-formula}. We call the realization of $\Phi$,
  namely the semi-algebraic set
  \begin{eqnarray*}
    \RR (\Phi, \R^{k}) & = & \{ \x \in \R^{k} \mid
    \Phi (\x) \}
  \end{eqnarray*}
  a \emph{$\mathcal{P}$-semi-algebraic set}. 
\item  \label{not:sign-condition}  We call an element $\sigma \in \{
  0,1,-1 \}^{\mathcal{P}}$, a \emph{sign condition} on $\mathcal{P}$. For
  any semi-algebraic set $Z \subset \R^{k}$, and a sign condition $\sigma \in
  \{ 0,1,-1 \}^{\mathcal{P}}$, we denote by $\RR (\sigma,Z)$ the
  semi-algebraic set defined by 
  \[
  \left\{ \x \in Z \mid \sign (P (\x)) = \sigma (P) ,P \in \mathcal{P} \right\},
  \]
  and call it the
  \emph{realization} of $\sigma$ on $Z$. 
  
\item We call a Boolean
  formula without negations, and with atoms 
  $P \;\sim\; 0, P\in \mathcal{P}$, and $\sim$ one of  $\{\leq,\geq\}$, 
  to be a \emph{$\mathcal{P}$-closed formula}, and we call
  the realization, 
  $\RR(\Phi, \R^{k})$, a \emph{$\mathcal{P}$-closed semi-algebraic set}.

  \item\label{not:sign-condition} For any finite family of polynomials $\mathcal{P}
  \subset \R [ X_{1}, \ldots,X_{k} ]$, we call an element $\sigma \in \{
  0,1,-1 \}^{\mathcal{P}}$, a \emph{sign condition} on $\mathcal{P}$. For
  any semi-algebraic set $Z \subset \R^{k}$, and a sign condition $\sigma \in
  \{ 0,1,-1 \}^{\mathcal{P}}$, we denote by $\RR (\sigma,Z)$ the
  semi-algebraic set defined by 
  \[
  \left\{ \x \in Z \mid \sign (P (\x)) = \sigma (P) ,P \in \mathcal{P} \right\},
  \]
  and call it the
  \emph{realization} of $\sigma$ on $Z$. 
  \item \label{7:def:sign}\label{7:not:sign}{\index{Sign condition!set of realizable
  sign conditions}} 
We denote by
  \[ \SIGN (\mathcal{P}) \subset {\{0,1, - 1\}^{\mathcal{P}}} \]
  the set of all  
  sign conditions $\sigma$ on $\mathcal{P}$ 
  such that $\RR(\sigma,\R^k) \neq \emptyset$. 
  We call $\SIGN(\mathcal{P})$ the set of \emph{realizable sign conditions of $\mathcal{P}$}.
  \end{enumerate}
\end{notation}

\begin{notation}
  For any semi-algebraic set 
  $X$, and a field of
  coefficients $\F$, we will denote by $\HH_{i} (X,\F)$ the
  \emph{$i$-th homology group} of $X$ with coefficients in
  $\F$, by $b_{i} (X,\F) =  \dim_{\F}  \HH_{i} (X,\F)$.
  \end{notation}
  
  Note here that we work over any real closed field. Therefore the definition of 
  homology groups
  is a little bit more delicate, in particular because $\R$ might be non-archimedean. 
  In case of a closed and bounded semi-algebraic set, $S$ the homology $\HH_{i} (S,\F)$ can be defined as the 
  $i$-th simplicial homology group associated to a semi-algebraic triangulation of $S$. The general case then is taken care of by constructing to  a general semi-algebraic set $S$ a semi-algebraic set $S'$, which is closed, bounded, and furthermore semi-algebraically homotopy equivalent to $S$.  We refer the reader to \cite[Chapter 6]{BPRbook3} for details of this construction. 
The  
topological
Euler-Poincar{\'e} characteristic of a
semi-algebraic set $S \subset \R^{k}$ is the alternating sum of the Betti
numbers of $S$. More precisely,

  \begin{eqnarray*}
    \chi^{\mathrm{top}}(S,\F) & = & \sum_{i} (-1)^{i} \dim_{\F}  \HH_{i}
    (S,\F).
  \end{eqnarray*}

For various applications (such as in  motivic integration \cite{Cluckers-Loeser} and other applications of
Euler integration \cite{Schapira91, Schapira95, Viro-euler,Ghrist2010})  the generalized Euler-Poincar{\'e} characteristic has proven to be more useful than the ordinary Euler-Poincar\'e  characteristic.
The main reason behind the usefulness of
the generalized Euler-Poincar{\'e} characteristic of a semi-algebraic set is its additivity
property, which is not satisfied by the topological Euler-Poincar{\'e} characteristic.
The generalized Euler-Poincar{\'e} characteristic agrees with the topological Euler-Poincar{\'e} characteristic
for compact semi-algebraic sets, but can be different for non-compact ones (see Example 
\ref{eg:Euler}). 
Nevertheless, the generalized Euler-Poincar{\'e} characteristic is intrinsically important because of the following reason.

The Grothendieck group $K_{0} (\mathbf{s}\mathbf{a}_\R)$ of semi-algebraic 
isomorphic 
classes  of semi-algebraic sets (two semi-algebraic sets being isomorphic if there is a continuous semi-algebraic bijection between them) 
(see for example \cite[Proposition 1.2.1]{Cluckers-Loeser}) is isomorphic to $\Z$, and the generalized Euler-Poincar{\'e} characteristic of a semi-algebraic set
can be identified with its image under the isomorphism that takes the class of a point (or any closed disk) to
$1$.

\begin{definition}
  \label{def:ep-general} The generalized Euler-Poincar{\'e} characteristic,
  $\chi^{\gen} (S)$, of a semi-algebraic set
  $S$ is uniquely defined by the following properties {\cite{Dries}}:
  \begin{enumerate}
    \item $\chi^{\tmop{gen}}$ is invariant under semi-algebraic
    homeomorphisms.
    
    \item $\chi^{\tmop{gen}}$ is multiplicative, i.e. $\chi^{\tmop{gen}} (A
    \times B) = \chi^{\tmop{gen}} (A) \cdot \chi^{\tmop{gen}} (B)$.
    
    \item $\chi^{\tmop{gen}}$ is additive, i.e. $\chi^{\tmop{gen}} (A \cup B
   ) = \chi^{\tmop{gen}} (A) + \chi^{\tmop{gen}} (B) - \chi^{\tmop{gen}}
    (A \cap B)$.
    
    \item $\chi^{\tmop{gen}} ([ 0,1 ]) =1$.
  \end{enumerate}
\end{definition}

The following examples are illustrative.
\begin{example}
\label{eg:Euler}
For every $n \geq 0$,
\begin{align*}
 \chi^{\tmop{gen}}([0,1]^n) &= \chi^{\tmop{gen}}([0,1])^n = 1, \\
 \chi^{\mathrm{top}}([0,1]^n)  &= 1,\\
 \chi^{\tmop{gen}}((0,1)^n)& = (\chi^{\tmop{gen}}(0,1))^n = (\chi^{\tmop{gen}}([0,1]) - \chi^{\tmop{gen}}({0}) - \chi^{\tmop{gen}}({1}))^n = (-1)^n, \\
 \chi^{\mathrm{top}}((0,1)^n) &= 1.
 \end{align*}
\end{example}

Let  $\mathfrak{S}_{{k}}$ denote the symmetric group on $k$-letters. Throughout the article we will consider the more general case of products of symmetric groups and we fix the following notation.

\begin{notation}\label{not:mult}
For $\omega\in\N$ and $\mathbf{k}= (k_{1}, \ldots,k_{\omega}) \in \Z_{>0}^{\omega}$, with $k= \sum_{i=1}^{\omega} k_{i}$ denote $\mathfrak{S}_{\mathbf{k}}=\mathfrak{S}_{{k_1}}\times\ldots\mathfrak{S}_{{k_\omega}}$. If $\omega =1$, then $k=k_{1}$, and we will denote $\mathfrak{S}_{\mathbf{k}}$
  simply by $\mathfrak{S}_{k}$.  
\end{notation}
A  set $X\subset\R^k$ is said to be symmetric, if it is closed under the action of $\mathfrak{S}_{\mathbf{k}}$. For such a set  we will denote by $X/\mathfrak{S}_{\mathbf{k}}$ the {{\em orbit space\/}} of this action.  
\begin{notation}
  \label{not:equivariant-betti} For any $\mathfrak{S}_{\mathbf{k}}$ symmetric
  semi-algebraic subset $S \subset \R^{k}$ with $\mathbf{k}= (k_{1}, \ldots
 ,k_{\omega}) \in \Z_{>0}^{\omega}$, with $k= \sum_{i=1}^{\omega} k_{i}$,
  and any field $\F$,  we denote
  
  \begin{enumerate}
  \item 
   $ \chi^{\mathrm{top}}(S,\F)  =  \sum_{i \geq 0} (-1)^{i}  b_{i} (
    S,\F),$
    \item $\chi_{\mathfrak{S}_{\mathbf{k}}} (S,\F)  =  \sum_{i \geq 0} (
    -1)^{i}  b_{\mathfrak{S}_{\mathbf{k}}}^{i} (S,\F),$

  \item $\chi_{\mathfrak{S}_{\mathbf{k}}}^{\gen} (S)
  = \chi^{\gen} (S/\mathfrak{S}_{\mathbf{k}}) =
  \chi^{\gen} (\phi_{\mathbf{k}} (S))$.
\end{enumerate}

\end{notation}

\subsection{Main result}
We describe new algorithms for computing the  \emph{ generalized Euler-Poincar{\'e} characteristic} (see Definition
\ref{def:ep-general}) of semi-algebraic sets defined in terms of symmetric polynomials. The algorithms we give here have complexity which is polynomial (for fixed degrees and the
number of blocks) in the number of symmetric variables. Since for systems of
equations with a finite set of solutions, the generalized Euler-Poincar{\'e}
characteristic of the set of solutions coincides with its cardinality, it is
easily seen that that computing the generalized Euler-Poincar{\'e}
characteristic of the set of solutions of a polynomial system with a fixed
degree bound is a $\#\mathbf{P}$-hard problem in general (i.e. in the
non-symmetric situation). Thus, this problem is believed to be unlikely to
admit a polynomial time solution.

We prove the following theorems.

\begin{theorem}
  \label{thm:algorithm-algebraic}Let $\D$ be an ordered domain contained in a
  real closed field $\R$. Then, there exists an algorithm that takes as input:
  \begin{enumerate}
    \item a tuple $\mathbf{k}= (k_{1}, \ldots,k_{\omega}) \in
    \Z_{>0}^{\omega}$, with $k= \sum_{i=1}^{\omega} k_{i}$;
    
    \item a  polynomial $P \in \D [ \X^{(1)}, \ldots
   ,\X^{(\omega)} ]$, where each $\X^{(i)}$ is a
    block of $k_{i}$ variables, and $P$ is symmetric in each block of
    variables $\X^{(i)}$;
  \end{enumerate}
  and computes the generalized Euler-Poincar{\'e} characteristics
  \[
\chi^{\gen} \left(\ZZ \left(P, \R^{k} \right)
  \right), 
  \chi_{\mathfrak{S}_{\mathbf{k}}}^{\gen} \left(
  \ZZ \left(P, \R^{k} \right) \right).
\]
 The complexity of the algorithm
  measured by the number of arithmetic operations in the ring $\D$ (including
  comparisons) is bounded by  $(\omega  k d)^{O (D)}$,
  where $d= \deg (P)$ and $D= \sum_{i=1}^{\omega} \min (k_{i},2d)$.
\end{theorem}

Notice that in case, $\omega =1$ and $\mathbf{k}= (k)$, the complexity is
polynomial in $k$ for fixed $d$.

We have the following result in the semi-algebraic case.

\begin{theorem}
  \label{thm:algorithm-sa}Let $\D$ be an ordered domain contained in a real
  closed field $\R$. Then, there exists an algorithm that takes as input:
  \begin{enumerate}
    \item a tuple $\mathbf{k}= (k_{1}, \ldots,k_{\omega}) \in
    \Z_{>0}^{\omega}$, with $k= \sum_{i=1}^{\omega} k_{i}$;
    
    \item a set of $s$ polynomials $\mathcal{P}= \{ P_{1}, \ldots,P_{s} \}
    \subset \D [ \X^{(1)}, \ldots,\X^{(\omega)} ]$,
    where each $\X^{(i)}$ is a block of $k_{i}$ variables, and
    each polynomial in $\mathcal{P}$ is symmetric in each block of variables
    $\X^{(i)}$ and of degree at most $d$;
    
    \item a $\mathcal{P}$-semi-algebraic set $S$, described by
    \begin{eqnarray*}
      S & = & \bigcup_{\sigma \in \Sigma} \RR \left(\sigma, \R^{k} \right),
    \end{eqnarray*}
    where $\Sigma \subset \{ 0,1,-1 \}^{\mathcal{P}}$;
  \end{enumerate}
  and computes the generalized Euler-Poincar{\'e} characteristics
 $
\chi^{\gen} (S),
  \chi_{\mathfrak{S}_{\mathbf{k}}}^{\gen} (S).
$
  The complexity of the algorithm measured by the number of arithmetic
  operations in the ring $\D$ (including comparisons) is bounded by
  \[ \card (
\Sigma)^{O (1)} + s^{D'} k^{d } d^{O (D'D'')} + s^{D'} d^{O(D'')} ( k \omega D)^{O (D''')}, \]
where 
$D=d  (D'' \log  d+ D'\log  s))$,
$D' = \sum_{i=1}^{\omega} \min (k_{i},d)$,
$D'' = \sum_{i=1}^{\omega} \min (k_{i},d)$, and
$D''' = \sum_{i=1}^\omega \min(k_i, 2D)$.

The algorithm also involves the inversion matrices of size $s^{D'} d^{O (D'')}$ with integer
coefficients.
\end{theorem}

Notice that the complexity in the semi-algebraic case is still polynomial in
$k$ for fixed $d$ and $s $ in the special case when $\omega =1$, and
$\mathbf{k}= (k)$.
Also note that, as a consequence of Proposition \ref{prop:SIGN} below, the number of sign conditions with non-empty realizations  in $\Sigma$ is bounded by $s^{D'} d^{O (D'')}$.

\begin{remark}
  An important point to note is that we give algorithms for computing
  both the ordinary as well as the equivariant 
  generalized 
  Euler-Poincar{\'e} 
  characteristics. For varieties or semi-algebraic sets defined by symmetric
  polynomials with degrees bounded by a constant, the ordinary
  generalized 
  Euler-Poincar{\'e} characteristic can be exponentially large in the
  dimension $k$. Nevertheless, our algorithms for computing it have
  complexities which are bounded polynomially in $k$ for fixed degree. 
\end{remark}

\subsection{Outline of the main techniques}

Efficient algorithms (with singly exponential complexity) 
for computing the Euler-Poincar\'e characteristics of semi-algebraic sets
\cite[Chapter 13]{BPRbook3}
usually proceed by first making a
deformation to a set defined by one inequality with smooth boundary and
non-degenerate critical points with respect to some affine function.
Furthermore, the new set is homotopy equivalent to the given variety and 
the Euler-Poincar\'e characteristic of this new set can be computed from certain local data 
in the neighborhood of each critical point (see \cite[Chapter 13]{BPRbook3} for more detail).
Since the number of critical points is at most singly exponential in number, such algorithms have a singly 
exponential complexity.

The approach
used in this paper for computing the Euler-Poincar\'e characteristics for symmetric
semi-algebraic sets  is  similar -- but differs on two  important points. 
Firstly, unlike in the general case, we are aiming here for an algorithm with polynomial
complexity (for fixed $d$). This requires that
the perturbation, as well as the Morse function both
need to be equivariant.  The choices are more restrictive (see Proposition
\ref{prop:non-degenerate}). 

Secondly, the topological changes at the
Morse critical points need to be analyzed more carefully (see Lemmas
\ref{lem:equivariant_morseA} and \ref{lem:equivariant_morseB}). The main
technical tool that makes the good dependence on the degree $d$ of the
polynomial possible is the so called ``half-degree principle''
{\cite{Riener,Timofte03}} (see 
Proposition \ref{prop:half-degree}), and this is what we use rather than
the Bezout bound to bound the number of (orbits of) critical points.
The proofs of these results appear in \cite{BC2013}, where they are used to prove bounds on the equivariant Betti numbers of semi-algebraic sets. 

Using these results, we
prove exact formulas for the ordinary as well as the equivariant
Euler-Poincar\'e characteristic of symmetric varieties (see \eqref{eqn:equivariant-ep1} and \eqref{eqn:equivariant-ep2} in Theorem \ref{thm:equivariant-ep}),
which form the basis of the algorithms described in this paper.

We adapt several non-equivariant algorithms from \cite{BPRbook3} to the
equivariant setting. The proofs of correctness of the algorithms described for
computing the ordinary as well as the equivariant (generalized)
Euler-Poincar{\'e} characteristics of algebraic as well as semi-algebraic
sets (Algorithms \ref{alg:ep-general-BM}, \ref{alg:ep-sign-conditions} and
\ref{alg:ep-sa}) follow from the equivariant Morse lemmas (Lemmas
\ref{lem:equivariant_morseA} and \ref{lem:equivariant_morseB}). The
complexity analysis follows from the complexities of similar algorithms in the
non-equivariant case \cite{BPRbook3}, but using the half-degree principle
referred to above. In the design of Algorithms
\ref{alg:ep-general-BM}, \ref{alg:ep-sign-conditions} and \ref{alg:ep-sa} we need
to use several subsidiary algorithms which are closely adapted from the
corresponding algorithms in the non-equivariant situation described in \cite{BPRbook3}. 
In particular, one of them, an algorithm for
computing the set of realizable sign conditions of a family of symmetric
polynomial (Algorithm \ref{alg:sampling}), whose complexity is
polynomial in the dimension for fixed degree could be of independent interest.

The rest of the paper is organized as follows. In \S \ref{sec:prelim}, we
recall certain facts from real algebraic geometry and topology that are needed
in the algorithms described in the paper. These include definitions of certain real
closed extensions of the ground field $\R$ consisting of algebraic Puiseux
series with coefficients in $\R$. We also recall some basic additivity properties of the
Euler-Poincar\'e characteristic.  In \S \ref{sec:deformation}, we define certain equivariant
deformations of symmetric varieties and state some topological properties of
these deformations, that mirror similar ones in the non-equivariant case. The proofs of these
properties appear in \cite{BC2013} and we give appropriate pointers where they can be found in that paper.  
In \S
\ref{sec:algorithms} we describe the algorithms for computing the
Euler-Poincar{\'e} characteristics of symmetric semi-algebraic sets proving
Theorems \ref{thm:algorithm-algebraic} and \ref{thm:algorithm-sa}.

\section{Mathematical Preliminaries}\label{sec:prelim}

In this section we recall some basic facts about real closed fields and real
closed extensions.

\subsection{Real closed extensions and Puiseux series}We will need some
properties of Puiseux series with coefficients in a real closed field. We
refer the reader to \cite{BPRbook3} for further details.

\begin{notation}
  For $\R$ a real closed field we denote by $\R \left\langle \eps
  \right\rangle$ the real closed field of algebraic Puiseux series in $\eps$
  with coefficients in $\R$. We use the notation $\R \left\langle \eps_{1},
  \ldots, \eps_{m} \right\rangle$ to denote the real closed field $\R
  \left\langle \eps_{1} \right\rangle \left\langle \eps_{2} \right\rangle
  \cdots \left\langle \eps_{m} \right\rangle$. Note that in the unique
  ordering of the field $\R \left\langle \eps_{1}, \ldots, \eps_{m}
  \right\rangle$, $0< \eps_{m} \ll \eps_{m-1} \ll \cdots \ll \eps_{1} \ll 1$.
\end{notation}

\begin{notation}
  For elements $x \in \R \left\langle \eps \right\rangle$ which are bounded
  over $\R$ we denote by $\lim_{\eps}  x$ to be the image in $\R$ under the
  usual map that sets $\eps$ to $0$ in the Puiseux series $x$.
\end{notation}

\begin{notation}
\label{not:extension}
  If $\R'$ is a real closed extension of a real closed field $\R$, and $S
  \subset \R^{k}$ is a semi-algebraic set defined by a first-order formula
  with coefficients in $\R$, then we will denote by $\Ext(S, \R') \subset \R'^{k}$ the semi-algebraic subset of $\R'^{k}$ defined by
  the same formula. It is well-known that $\Ext(S, \R')$ does
  not depend on the choice of the formula defining $S$ \cite{BPRbook3}.
\end{notation}

\begin{notation}
\label{not:ball}
  For $\x \in \R^{k}$ and $r \in \R$, $r>0$, we will denote by $B_{k} (\x,r)$
  the open Euclidean ball centered at $\x$ of radius $r$, and we denote by
  $S^{k-1}(\x,r)$ the sphere of radius $r$ centered at $\x$.   If $\R'$ is a real
  closed extension of the real closed field $\R$ and when the context is
  clear, we will continue to denote by $B_{k} (\x,r)$ (respectively, $S^{k-1}(\x,r)$) the extension $\Ext(B_{k} (\x,r), \R')$ (respectively, $\Ext(S^{k-1}(\x,r),\R')$). This should not cause any confusion.
\end{notation}

\subsection{Tarski-Seidenberg transfer principle}
In some proofs that involve
Morse theory (see for example the proof of Lemma \ref{lem:equivariant_morseB}), where integration of gradient flows is used in an essential way, we
first restrict to the case $\R =\mathbb{R}$. After having proved the result
over $\mathbb{R}$, we use the Tarski-Seidenberg transfer theorem to extend
the result to all real closed fields. We refer the reader to 
\cite[Chapter 2]{BPRbook3} for an exposition of the Tarski-Seidenberg transfer
principle.

\subsection{Additivity property of the Euler-Poincar\'e characteristics}

We need the following additivity property of the Euler-Poincar\'e characteristics that
follow from the Mayer-Vietoris exact sequence.

\begin{proposition}
  \label{prop:MV}If $S_{1},S_{2}$ are closed semi-algebraic sets, then for
  any field $\F$ and every $i \geq 0,$
  \begin{eqnarray}
   \chi^{\mathrm{top}}(S_{1} \cup S_{2},\F) & = & \chi^{\mathrm{top}}(S_{1},\F) +
    \chi^{\mathrm{top}}(S_{2},\F) - \chi^{\mathrm{top}}(S_{1} \cap S_{2},\F). 
    \label{eqn:MV2-ep}
  \end{eqnarray}
\end{proposition}

\begin{proof}
See for example
\cite[Proposition 6.36]{BPRbook3}.
\end{proof}

We also recall the definition of the
Borel-Moore homology groups of locally closed semi-algebraic sets and some of its properties.

\subsection{Borel-Moore homology groups}
\begin{definition}
  \label{def:BM}Let $S \subset \R^{k}$ be a locally closed semi-algebraic
  set and let $S_{r} =S \cap B_{k} (0,r)$. The \emph{$p$-th Borel-Moore
  homology group} of $S$ with
  coefficients in a field $\F$, denoted by
  $\HH_{p}^{ \mathrm{BM}} (S,\F)$,\label{6:not-28} is
  defined to be the $p$-th simplicial homology group of the pair $\left(
  \overline{S_{r}}, \overline{S_{r}} \setminus S_{r} \right)$ with
  coefficients in $\F$, for large enough $r>0$. 
\end{definition}

\begin{notation}
  \label{not:ep-BM}For any locally closed semi-algebraic set $S$ we denote
  \begin{eqnarray*}
    \chi^{ \mathrm{BM}} (S,\F) & = & \sum_{i \geq
    0} (-1)^{i}   \dim_{\F}  \HH_{i}^{ \mathrm{BM}}
    (S,\F).
  \end{eqnarray*}
\end{notation}

It follows immediately from the exact sequence of the homology of the pair
$\left(\overline{S_{r}}, \overline{S_{r}} \setminus S_{r} \right)$ that

\begin{proposition}
  \label{prop:BM-pair}If $S$ is a locally closed semi-algebraic set then for
  all $r>0$ large enough
  \begin{eqnarray*}
    \chi^{ \mathrm{BM}} (S,\mathbb{Q}) & = & \chi^{\mathrm{top}} \left(
    \overline{S_{r}},\mathbb{Q} \right) - \chi^{\mathrm{top}}(S \cap S^{k-1} (0,r)
   ,\mathbb{Q}).
  \end{eqnarray*}
\end{proposition}

It follows from the fact that $\chi^{\mathrm{BM}}(\cdot,\Q)$ is additive for locally closed semi-algebraic sets
(cf. \cite[Proposition 6.60]{BPRbook3}), and the uniqueness of the valuation $\chi^{\mathrm{gen}}(\cdot)$ that:

\begin{proposition}
  \label{prop:ep-gen-BM}
  If $S$ is a locally closed semi-algebraic set, then
  \begin{eqnarray*}
    \chi^{\gen} (S) & = &
    \chi^{ \mathrm{BM}} (S,\mathbb{Q}).
  \end{eqnarray*}
  Moreover, if $S$ is a closed and bounded semi-algebraic set then,
  \[
   \chi^{\gen} (S) =  \chi^{ \mathrm{BM}} (S,\mathbb{Q}) = \chi^{\mathrm{top}}(S,\mathbb{Q}).
   \]
    \end{proposition}

The following proposition is an immediate consequence of Definition
\ref{def:BM}, Notation \ref{not:ep-BM} and Propositions \ref{prop:BM-pair} and
\ref{prop:ep-gen-BM}.

\begin{proposition}
  \label{prop:ep-unbounded} Let $S \subset \R^{k}$ be a closed semi-algebraic
  set.Then,
  \begin{eqnarray*}
    \chi^{\gen} (S) & = &
    \chi^{\gen} \left(S \cap \overline{B_{k} (0,r
   )} \right) - \chi^{\gen} (S \cap S^{k-1} (0,r
   ))
  \end{eqnarray*}
  for all large enough $r>0$. 
\end{proposition}

\begin{proof}By the theorem on conic structure of
semi-algebraic sets at infinity (see \cite[Proposition 5.49]{BPRbook3}) we have that $S$ is
semi-algebraically homeomorphic to $S \cap B_{k} (0,r)$ for all large enough
$r>0$. Also, note that $S \cap \overline{B_{k} (0,r)}$ is a disjoint union of
$S \cap B_{k} (0,r)$ and $S \cap S^{k-1} (0,r)$. The proposition follows
from the additivity of $\chi^{\tmop{gen}} (\cdot
)$.\end{proof}

\begin{corollary}
  \label{cor:additive}Let $S \subset \R^{k}$ be a $\mathcal{P}$-closed
  semi-algebraic set. Let $\Gamma \subset \{ 0,1,-1 \}^{\mathcal{P}}$ be the
  set of realizable sign conditions $\gamma$ on $\mathcal{P}$ such that 
  $\RR
  \left(\gamma, \R^{k} \right) \subset S.
  $
  Then,
  \begin{eqnarray*}
    \chi^{\gen} (S) & = & \sum_{\gamma \in
    \Gamma} \chi^{\gen} \left(\RR \left(\gamma,
    \R^{k} \right) \right).
  \end{eqnarray*}
  
\end{corollary}

\begin{proof}Clear from the definition of the generalized
Euler-Poincar{\'e} characteristic (Definition
\ref{def:ep-general}).\end{proof}

\section{Equivariant deformation\label{sec:deformation}}

In this section we recall the definition  of certain equivariant
deformations of symmetric real algebraic varieties  that were introduced in \cite{BC2013}. These are adapted from the
non-equivariant case (see for example \cite{BPRbook3}), but keeping
everything equivariant requires additional effort. 

\begin{notation}
For $i\in\N$ let $p_{i}^{(k)}:=\sum_{j=1}^{(k)} X_{j}^i$ denote the $i$-th Newton sum and 
  \label{not:def}for any $P \in \R [ X_{1}, \ldots,X_{k} ]$ we denote
  \[ \Def (P, \zeta,d) = P -  \zeta   \left(1+ p_{d}^{(k)}
     \right), \]
  where $\zeta$ is a new variable.
\end{notation}

Notice that if $P$ is symmetric in $X_{1}, \ldots,X_{k}$, so is $\Def (P,
\zeta,d)$.

\subsection{Properties of $\Def(P,\zeta,d)$}

We now state some key properties of the deformed polynomial $\Def(P,\zeta,d)$ that will be important
in proving the correctness, as well as the complexity analysis, of the algorithms presented later in the paper.
Most of these properties, with the exception of the key Theorem \ref{thm:equivariant-ep}, have been proved in \cite{BC2013} and we refer the reader to that paper for the proofs. We reproduce the statements below for ease of reading and completeness of the current paper.
 
\begin{proposition}\cite[Proposition 3]{BC2013}
  \label{prop:alg-to-semialg}
  Let $d \geq 0$ be even, $\mathbf{k}= (k_{1}, \ldots,k_{\omega}) \in \Z_{>0}^{\omega}$, with $k= \sum_{i=1}^{\omega} k_{i}$, 
  and $P \in \R [ \X^{(1)}, \ldots,\X^{(\omega)} ]_{\leq d}$, where each $\X^{(i)}$ is a
  block of $k_{i}$ variables, such that $P$ is non-negative and symmetric in
  each block of variable $\X^{(i)}$.  Also
  suppose that $V = \ZZ(P, \R^{k})$ is bounded.
Then,
  $\Ext(V, \R \langle \zeta \rangle^{k})$ is a semi-algebraic
  deformation retract of the (symmetric) semi-algebraic subset $S$ of $\R
  \langle \zeta \rangle^{k}$,
consisting of the union of the semi-algebraically connected components of the semi-algebraic set  
  defined by the inequality 
  $\Def (P, \zeta,d) \leq 0$,
  which are bounded over $\R$.
  Hence, $\Ext(V,\R\la\zeta\ra)$ is semi-algebraically homotopy equivalent to $S$.
  Moreover, $\phi_{\mathbf{k}} (\Ext(V, \R \langle \zeta \rangle^{k}))$
  is semi-algebraically homotopy  equivalent to $\phi_{\mathbf{k}} (S)$. 
\end{proposition}

\begin{proposition}\cite[Proposition 4]{BC2013}
  \label{prop:non-degenerate}Let $P \in \R [ X_{1}, \ldots,X_{k} ]  $, and
  $d$ be an even number with $\deg (P) <  d=p+1$, with $p$ a prime. Let
  $F=p_{1}^{(k)}(X_{1}, \ldots,X_{k})$.
  Let 
  \[
  V_{\zeta} = \ZZ \left(\Def (P,\zeta,d), \R \langle \zeta \rangle^{k} \right).
  \] 
  Suppose also that $\gcd (p,k) =1$. Then, the critical points of $F$ restricted to $V_{\zeta}$ are
  finite in number, and each critical point is non-degenerate.
\end{proposition}

\begin{notation}
  For any pair $(\mathbf{k},
  \boldsymbol{\ell}
  $, where $\mathbf{k}= (k_{1},
  \ldots,k_{\omega}) \in \Z_{>0}^{\omega}$, $k= \sum_{i=1}^{\omega}
  k_{i}$, and $\boldsymbol{\ell}= (\ell_{1}, \ldots, \ell_{\omega})$, with $1
  \leq \ell_{i} \leq k_{i}$, we denote by $A_{\mathbf{k}}^{\boldsymbol{\ell}}$ the
  subset of $\R^{k}$ defined by
  \begin{eqnarray*}
    A^{\boldsymbol{\ell}}_{\mathbf{k}} & = & \left\{ x= (x^{(1)}, \ldots x^{(
    \omega)}) \mid  \card \left(
    \bigcup_{j=1}^{k_{i}} \{ x_{j}^{(i)} \} \right) = \ell_{i} \right\}.
  \end{eqnarray*}
\end{notation}

\begin{proposition}\cite[Proposition 5]{BC2013}
  \label{prop:half-degree}Let $\mathbf{k}= (k_{1}, \ldots,k_{\omega})
  \in \Z_{>0}^{\omega}$, with 
 $k= \sum_{i=1}^{\omega} k_{i}$,
    and 
    \[P
  \in \R [ \X^{(1)}, \ldots
 ,\X^{(\omega)} ],
  \] 
  where each
  $\X^{(i)}$ is a block of $k_{i}$ variables, such
  that $P$ is non-negative and symmetric in each block of variable
  $\X^{(i)}$ and $\deg (P) \leq d$. Let $(X_{1}
 , \ldots,X_{k})$ denote the set of variables $(
  \X^{(1)}, \ldots,\X^{(
  \omega)})$ and let $F=p_{1}^{(k)} (X_{1}, \ldots,X_{k})$. Suppose
  that the critical points of $F$ restricted to $V= \ZZ \left(P, \R^{k}
  \right)$ are isolated. Then, each critical point of $F$ restricted to $V$ is
  contained in $A^{\boldsymbol{\ell}}_{\mathbf{k}}$ for some $\boldsymbol{\ell}= (\ell_{1}
 , \ldots, \ell_{\omega})$ with each $\ell_{i} \leq d$.
\end{proposition}

With the same notation as in Proposition \ref{prop:half-degree}:
\begin{corollary}
\label{cor:half-degree}
Let $\mathcal{P} \subset \R[\X^{(1)},\ldots,\X^{(k_\omega)}]$ be a finite set of polynomials,
such that for each $P \in \mathcal{P}$,
$P$ is non-negative and symmetric in each block of variable
  $\X^{(i)}$, and $\deg (P) \leq d$.
 Let $C$ be a bounded semi-algebraically connected component of $\ZZ(\mathcal{P}, \R^k)$. Then, $C \cap A^{\boldsymbol{\ell}}_{\mathbf{k}} \neq \emptyset$, 
for some $\boldsymbol{\ell}= (\ell_{1}
 , \ldots, \ell_{\omega})$,  where for each $i, 1 \leq i \leq \omega$,  $1 \leq \ell_{i} \leq 2d$.
 \end{corollary}

\begin{proof}
Let $d'$ be
the least even number such that $d' >d$ and such that $d' -1$ is
prime. By Bertrand's postulate we have that $d'   \leq 2d$. Now, if $p$
divides $k$, replace each $P \in \mathcal{P}$ by the polynomial
\[ P  \noplus \noplus +X_{k+1}^{2},  \]
and let $\omega' = \omega +1$, $k' =k+1$, and $\mathbf{k}' = (\mathbf{k},1
)$. Otherwise, let $\omega' = \omega +1$, $k' =k$, and $\mathbf{k}' = (
\mathbf{k},0)$. In either case, we have that $\gcd (p,k') =1$, and $k'
\leq k+1$.

Let $Q = \sum_{P \in \mathcal{P}} P$, and let $V_\zeta = \ZZ(\Def(Q,\zeta,d'),\R\la\zeta\ra^{k'})$.
Then for every bounded semi-algebraically connected component $C$ of $\ZZ(Q,\R^{k'})$,
there exists a semi-algebraically connected component of $C_\zeta$ of $V_\zeta$ bounded over $\R$,
such that $\lim_\zeta C_\zeta \subset C$ (see \cite[Proposition 12.51]{BPRbook3}).  
Now  every bounded semi-algebraically connected component $C_\zeta$ of $V_\zeta$ contains at least two critical points
of the polynomial $e^{(k)}_1$ restricted to $V_\zeta$, and they are isolated by 
Proposition \ref{prop:non-degenerate}.
The corollary now follows from Proposition
\ref{prop:half-degree}.
\end{proof}

The next theorem which gives an exact expression for both $\chiep(S,\F)$ as well as
$\chiep(S/\mathfrak{S}_\kk,F)$ (where $S$ is as in Proposition \ref{prop:alg-to-semialg})
is the key result needed for the algorithms in the paper. 
We defer its proof to the appendix.

Before stating the theorem  we need to introduce a few more  notation.

\begin{notation}
  \label{not:partition}(Partitions) We denote by $\Pi_{k}$ the set of
  partitions of $k$, where each partition $\pi = (\pi_{1}, \pi_{2}, \ldots
 , \pi_{\ell}) \in \Pi_{k}$, where $\pi_{1} \geq \pi_{2} \geq \cdots \geq
  \pi_{\ell} \geq 1$, and $\pi_{1} + \pi_{2} + \cdots + \pi_{\ell} =k$. We
  call $\ell$ the length of the partition $\pi$, and denote
  $\length (\pi) = \ell$. 
  
  More generally, for any tuple $\mathbf{k}= (k_{1}, \ldots,k_{\omega})
  \in \Z_{>0}^{\omega}$, we will denote by
  $\boldPi_{\mathbf{k}} = \Pi_{k_{1}} \times \cdots
  \times \Pi_{k_{\omega}}$, and for each $\boldpi= (
  \pi^{(1)}, \ldots, \pi^{(\omega)}) \in
  \boldPi_{\mathbf{k}}$, we denote by
  $\length (\boldpi) =
  \sum_{i=1}^{\omega} \length (\pi^{(i)})$. We
  also denote for each $\boldsymbol{\ell}= (\ell_{1}, \ldots, \ell_{\omega}) \in
  \Z_{>0}^{\omega}$,
  \begin{eqnarray*}
    | \boldsymbol{\ell} | & = & \ell_{1} + \cdots + \ell_{\omega}.
    \end{eqnarray*}
\end{notation}

\begin{notation}
\label{noy:L-pi}
  Let $\boldpi \in
  \boldPi_{\mathbf{k}}$ where $\mathbf{k}= (k_{1},
  \ldots,k_{\omega}) \in \Z_{>0}^{\omega}$, with $k=
  \sum_{i=1}^{\omega} k_{i}$. 
  
  For $1 \leq i  \leq \omega$, and $1 \leq j \leq
  \length (\pi^{(i)})$, let $L_{\pi^{(i
 )}_{j}} \subset \R^{k}$ be defined by the equations
  \begin{eqnarray*}
    X^{(i)}_{\pi_{1}^{(i)} + \cdots + \pi_{j-1}^{(i)} +1} & = \cdots = &
    X^{(i)}_{\pi_{1}^{(i)} + \cdots + \pi_{j}^{(i)}},
  \end{eqnarray*}
and let
  \begin{eqnarray*}
    L_{\boldpi} & = & \bigcap_{1 \leq i \leq \omega}
    \bigcap_{1 \leq j \leq \length (\pi^{(i)})}
    L_{\pi^{(i)}_{j}}. 
  \end{eqnarray*}
  \end{notation}
\begin{notation}

\label{not:L-fixed}
Let $L \subset \R^{k}$ be the subspace defined by
  $\sum_{i} X_{i} =0$, and $\boldpi= (\pi^{(1)},
  \ldots, \pi^{(\omega)}) \in \boldPi_{\mathbf{k}}$.
  Let for each $i$, $1 \leq i \leq \omega$, $\pi^{(i)} = (\pi^{(i)}_{1},
  \ldots, \pi^{(i)}_{\ell_{i}})$, and for each $j,1 \leq j \leq \ell_{i}
 ,$ let $L^{(i)}_{j}$ denote the subspace $L \cap L_{\pi^{(i)}_{j}}$ of
  $L$, and $M^{(i)}_{j}$ the orthogonal complement of $L^{(i)}_{j}$ in
  $L$. We denote 
  \[
  L_{\fixed} =L \cap  L_{\boldpi}.
  \]
\end{notation}

We  have the following theorem which gives  an exact expression for the
Euler-Poincar{\'e} characteristic of a symmetric semi-algebraic set defined by
one polynomial inequality satisfying the same conditions as in Lemmas
\ref{lem:equivariant_morseA} and \ref{lem:equivariant_morseB} above. The proof of the theorem which depends on the properties of $\Def(P,\zeta,d)$ stated above is given
in \S \ref{sec:appendix}.

\begin{theorem}
  \label{thm:equivariant-ep}Let $\mathbf{k}= (k_{1}, \ldots,k_{\omega})
  \in \Z_{>0}^{\omega}$, with $k= \sum_{i=1}^{\omega} k_{i}$, and let
  $S \subset \R^{k}$ be a bounded symmetric basic semi-algebraic set defined
  by $P \leq 0$, where 
  $P \in \R [ \X^{(1)},\ldots,\X^{(\omega)} ]^{\mathfrak{S}_\kk}$.
  where each  $\X^{(i)}$ is a block of $k_{i}$ variables.

  Let $W= \ZZ (P, \R^{k})$ be non-singular and bounded. 
  Let $(X_{1}, \ldots,X_{k})$ denote the variables $(\X^{(1)}, \ldots
 ,\X^{(\omega)})$ and suppose that $F=p_{1}^{(k)}(X_{1}, \ldots,X_{k})$ restricted to $W$ has a finite number of critical
  points, all of which are non-degenerate. Let $C$ be the finite set of
  critical points $\x$ of $F$ restricted to $W$ such that $\sum_{1 \leq i \leq
  k} \dfrac{\partial P}{\partial X_{i}} (\x)  <0$, and let
  $\Hess (\x)$ denote the Hessian of $F$
  restricted to W at $\x$. Then, for any field of coefficients $\F$,
  \begin{eqnarray}
   \label{eqn:equivariant-ep1}
   &&\\
   \nonumber
    \chi^{\mathrm{top}}(S,\F) & = & \sum_{\boldpi= (\pi^{(
    1)}, \ldots, \pi^{(\omega)}) \in
    \boldPi_{\mathbf{k}}} \sum_{\x  \in C \cap
    L_{\boldpi}} (-1
   )^{\ind^{-} (\Hess
    (\x))} \binom{k_{1}}{\pi^{(1)}} \cdots \binom{k_{\omega}}{\pi^{(
    \omega)}}, \\
    \label{eqn:equivariant-ep2}
    && \\
\nonumber
    \chi_{\mathfrak{S}_{k}} (S,\F) & = &
    \sum_{\boldpi \in
    \boldPi_{\mathbf{k}}} \sum_{\x  \in C \cap
    L_{\boldpi},L^{-} (\x) \subset
    L_{\fixed}} (-1
   )^{\ind^{-} (\Hess
    (\x))},  
  \end{eqnarray}
  (where for $\pi = (\pi_{1}, \ldots, \pi_{\ell}) \in \Pi_{k}$,
  $\binom{k}{\pi}$ denotes the multinomial coefficient $\binom{k}{\pi_{1},
  \ldots, \pi_{\ell}}$).
\end{theorem}

\begin{proof} 

See \S \ref{sec:appendix} (Appendix).
\end{proof}

Theorem \ref{thm:equivariant-ep} is illustrated by the following simple
example.

\begin{figure}[h]
 
  \includegraphics[scale=0.60]{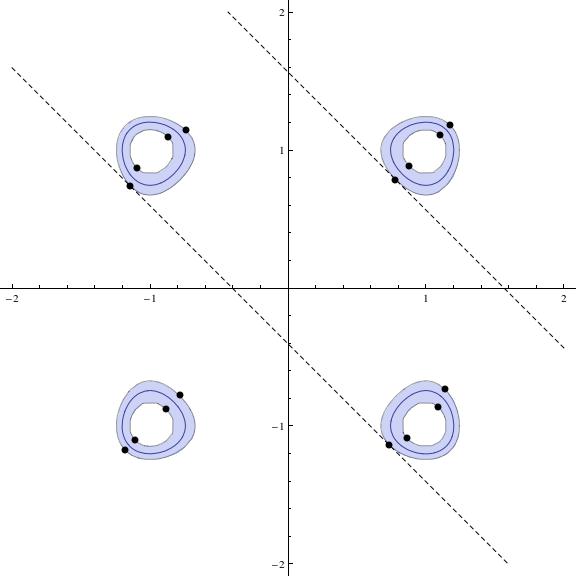}
\caption{\label{fig:figure1}}
\end{figure}
\begin{example}
  In this example, the number of blocks $\omega =1$, and $k=k_{1} =2$. \
  Consider the polynomial
  \begin{eqnarray*}
    P & = & (X_{1}^{2} -1)^{2} + (X_{2}^{2} -1)^{2} - \eps,
  \end{eqnarray*}
  for some small $\eps >0$. The sets $\ZZ \left(P, \R^{2} \right)$, and $S=
  \left\{ x \in \R \langle \zeta \rangle^{2} \mid \bar{P} \leq 0 \right\}$,
  where $\bar{P} =  \Def (P, \zeta,6)$ is shown in the Figure
  \ref{fig:figure1}.
  
  The polynomial $p_{1}^{(2)}(X_{1},X_{2}) =X_{1} +X_{2}$ has $16$ critical
  points, corresponding to $12$ critical values, $v_{1} < \cdots <v_{12}$, on
  $\ZZ \left(\bar{P}, \R \langle \zeta \rangle^{2} \right)$ of which $v_{5}$
  and $v_{9}$ are indicated in Figure \ref{fig:figure1} using dotted lines.
  The corresponding indices of the critical points, the number of critical
  points for each critical value, the sign of the polynomial $\dfrac{\partial
  \bar{P}}{\partial X_{1}} + \dfrac{\partial \bar{P}}{\partial X_{2}}$ at
  these critical points, and the partition $\pi \in \Pi_{2}$ such that the
  corresponding critical points belong to $L_{\pi}$ are shown in Table
  \ref{tab:table1}. The critical points corresponding to the shaded rows are
  then the critical points where $\left(\dfrac{\partial \bar{P}}{\partial
  X_{1}} + \dfrac{\partial \bar{P}}{\partial X_{2}} \right) <0$, and these are
  the critical points which contribute to the sums in Eqns.
  \eqref{eqn:equivariant-ep1} and \eqref{eqn:equivariant-ep2}.
  
  {\tiny
  \begin{table}[h]
    \begin{tabular}{|l|l|l|l|l|l|l|}
      \hline
      Critical values & Index  & $ \SIGN \left(
      \dfrac{\partial \bar{P}}{\partial X_{1}} + \dfrac{\partial
      \bar{P}}{\partial X_{2}} \right)$ & $\pi$ & $L^{-} (p)$ &
      $L_{\fixed}$ & $L^{-} (p) \subset
      L_{\fixed}$\\
      \hline
      \rowcolor{LightCyan}$v_{1}$ & 0 & $-1$ & $(2)$ & $0$ & $0$ & yes\\
      \hline
      $v_{2}$ & 0 & $1$ & $(2)$ & 0 & $0$ & yes\\
      \hline\rowcolor{LightCyan}
      $v_{3}$ & 1 & $-1$ & $(2)$ & $L$ & $0$ & no\\
      \hline
      $v_{4}$ & 1 & $1$ & $(2)$ & $L$ & $0$ & no\\
      \hline\rowcolor{LightCyan}
      $v_{5}$ & $0$ & $-1$ & $(1,1)$ & $0$ & $L$ & yes\\
      \hline
      $v_{6}$ & 0 & 1 & $(1,1)$ & $0$ & $L$ & yes\\
      \hline\rowcolor{LightCyan}
      $v_{7}$ & $1$ & $-1$ & $(1,1)$ & $L$ & $L$ & yes\\
      \hline
      $v_{8}$ & $1$ & $1$ & $(1,1)$ & $L$ & $L$ & yes\\
      \hline\rowcolor{LightCyan}
      $v_{9}$ & $0$ & $-1$ & $(2)$ & $0$ & $0$ & yes\\
      \hline
      $v_{10}$ & 0 & $1$ & $(2)$ & $0$ & $0$ & yes\\
      \hline\rowcolor{LightCyan}
      $v_{11}$ & $1$ & $-1$ & $(2)$ & $L$ & $0$ & no\\
      \hline
      $v_{12}$ & $1$ & 1 & $(2)$ & $L$ & $0$ & no\\
      \hline
    \end{tabular}
    \caption{\label{tab:table1}}
  \end{table}
  }

  It is now easy to verify using Eqns. \eqref{eqn:equivariant-ep1} and
  \eqref{eqn:equivariant-ep2} that,
  \begin{eqnarray*}
    \chi^{\mathrm{top}} \left(\ZZ \left(P, \R^{k} \right),\mathbb{Q} \right) & = & \chi^{\mathrm{top}}(
    S,\mathbb{Q})\\
    &= & (-1)^{0} \binom{2}{2} + (-1)^{1} \binom{2}{2} + (-1)^{0}
    \binom{2}{1,1} \\
&& \;\; +\;\; (-1)^{1} \binom{2}{1,1} + (-1)^{0} \binom{2}{2} + (
    -1)^{1} \binom{2}{2}\\
    & = & 1-1+2-2+1-1=0.\\
    \chi_{\mathfrak{S}_{2}} \left(\ZZ \left(P, \R^{k} \right),\mathbb{Q}
    \right) & = & \chi_{\mathfrak{S}_{2}} (S,\mathbb{Q})\\
    & = & (-1)^{0} + (-1)^{0} + (-1)^{1} + (-1)^{0}\\
    & = & 2.
  \end{eqnarray*}
\end{example}

\section{Algorithms and the proofs of the main theorems}
\label{sec:algorithms}

In this section we describe new algorithms for computing the (generalized)
Euler-Poincar{\'e} characteristic of symmetric semi-algebraic subsets of
$\R^{k}$, prove their correctness and analyze their complexities. 
As a consequence we prove Theorems \ref{thm:algorithm-algebraic} and 
\ref{thm:algorithm-sa}.

We first  recall some basic algorithms from \cite{BPRbook3} which we will need
as subroutines in our algorithms.

\subsection{Algorithmic Preliminaries}

In this section we recall the input, output and complexities of some basic
algorithms and also some notations from the book \cite{BPRbook3}. These
algorithms will be the building blocks of our main algorithms described later.

\begin{definition}
  \label{2:def:Thom encoding}Let $P \in \R [X]$ and $\sigma \in \{0,1, -
  1\}^{\Der (P)}$, a sign condition on the set $\Der (P)$ of derivatives of
  $P$. The sign condition $\sigma$ is\label{2:def:thom}{\index{Thom encoding}}
  a \emph{Thom encoding} of $x \in \R$ if $\sigma (P) =0$ and $\RR (
  \sigma) = \{x\}$, i.e. $\sigma$ is the sign condition taken by the set
  $\Der (P)$ at $x$. 
\end{definition}

\begin{notation}
  \label{12:def:rurassociated}A \emph{$k$-univariate
  representation} $u$ is a $k+2$-tuple of
  polynomials in $\R[T]$,
  \[ u= (f(T),g(T)), \ensuremath{\operatorname{with}} g= (g_{0} (T),g_{1}
     (T), \ldots,g_{k} (T)), \]
  such that $f$ and $g_{0}$ are co-prime. Note that $g_{0} (t) \neq 0$ if $t
  \in \C$ is a root of $f (T)$. The points \emph{associated}
to a univariate representation $u$
  are the points
  \begin{equation}
    x_{u} (t) = \left(\frac{g_{1} (t)}{g_{0} (t)}, \ldots, \frac{g_{k}
    (t)}{g_{0} (t)} \right) \in \C^{k} \text{\label{12:eq:assoc}}
  \end{equation}
  where $t \in \C$ is a root of $f (T)$.
  
  Let $\mathcal{P} \subset \R [X_{1}, \ldots,X_{k} ]$ be a finite set of
  polynomials such that $\ZZ (\mathcal{P}, \C^{k})$ is finite. The
  $k+2$-tuple $u= (f(T),g(T))$, {\emph{represents}} $\ZZ (\mathcal{P},
  \C^{k})$ 
if $u$ is a univariate representation and
  \[ \ZZ (\mathcal{P}, \C^{k}) \text{=} \left\{ x \in \C^{k} | \exists
     t \in \ZZ (\mathcal{f}, \C) x=x_{u} (t) \right\}. \]
  A {\emph{real k-univariate representation}} is a pair $u, \sigma$ where $u$ is a
  $k$-univariate representation and $\sigma$ is the Thom encoding of a root of
  $f$, $t_{\sigma} \in{\R}$. The point
  \emph{associated} to the real univariate representation $u, \sigma$ is the point
  
 \begin{equation}
    x_{u} (t_{\sigma}) = \left(\frac{g_{1} (t_{\sigma})}{g_{0} (t_{\sigma}
   )}, \ldots, \frac{g_{k} (t_{\sigma})}{g_{0} (t_{\sigma})} \right) \in
    \R^{k}.
  \end{equation}
\end{notation}

For the rest of this section we fix an ordered domain $\D$ contained in the
real closed field $\R$. By complexity of an algorithm whose input consists of
polynomials with coefficients in $\D$, we will mean (following
\cite{BPRbook3}) the maximum number of arithmetic operations in $\D$
(including comparisons) used by the algorithm for an input of a certain size.

We will use four algorithms from the book \cite{BPRbook3}:
namely,  Algorithm 10.98 (Univariate Sign Determination), Algorithm 12.64 (Algebraic Sampling), Algorithm 12.46 (Limit of Bounded Points),
and  Algorithm 10.83 (Adapted Matrix).
We refer the reader to 
\cite{BPRbook3}
for the descriptions of these algorithms and their complexity analysis.

\subsection{Computing the generalized Euler-Poincar{\'e} characteristic of
symmetric real algebraic sets}

We now describe our algorithm for computing the generalized Euler-Poincar{\'e}
characteristic for real varieties, starting as usual with the bounded case.
Note that using  Proposition  \ref{prop:ep-gen-BM}, for a closed and bounded semi-algebraic set $S$,
\begin{eqnarray*}
  \chi^{\gen} (S) & = & \chi^{\mathrm{top}}(S,\mathbb{Q}).
\end{eqnarray*}

\vspace{-0.1in}
{\small
\begin{algorithm}[H]
\caption{(Generalized Euler-Poincar{\'e} characteristic for bounded
symmetric algebraic sets)}
\label{alg:ep-bounded}

\begin{algorithmic}[1]
\INPUT
\Statex{
\begin{enumerate}
 \item A tuple $\mathbf{k}= (k_{1}, \ldots,k_{\omega}) \in
    \Z_{>0}^{\omega}$, with $k= \sum_{i=1}^{\omega} k_{i}$.
\item A polynomial $P \in \D [ \X^{(1)}, \ldots
   ,\X^{(\omega)} ]$, where each $\X^{(i)}$ is a
    block of $k_{i}$ variables, and $P$ is non-negative, symmetric in each block of
    variables $\X^{(i)}$, and such that $\ZZ \left(P, \R^{k}
    \right)$ is bounded and of degree at most $d$.
\end{enumerate}
}
  
\OUTPUT
\Statex{
  $\chi^{\tmop{gen}} \left(\ZZ \left(P, \R^{k} \right) \right)$
  and 
  $\chi^{\tmop{gen}}_{\mathfrak{S}_{\mathbf{k}}} \left(\ZZ \left(P, \R^{k}
  \right) \right)$.
  }

  \PROCEDURE
  
 \State{ Pick $d'$, such that $d < d' \leq 2d$, $d'$ even, and $d'
    -1$ a prime number, by sieving through all possibilities, and testing for
    primality using the naive primality testing algorithm (i.e. testing for
    divisibility using the Euclidean division algorithm for each possible
    divisor).
    } \label{alg:ep-bounded:step1}
    
    \If {$d' -1 \mid k$} 
          \State{ $P \gets P+X_{k+1}^{2}$,
    $\mathbf{k} \gets \mathbf{k}' = (\mathbf{k},1)$, and $k \gets k+1$.}\label{alg:ep-bounded:step2}
    \EndIf

    \State{  $Q \gets \Def (P, \zeta,d')$.} \label{alg:ep-bounded:step3}

    \State{  $\chi^{\tmop{gen}} \gets 0$,
    $\chi_{\mathfrak{S}_{\mathbf{k}}}^{\tmop{gen}}  \gets  0$.} \label{alg:ep-bounded:step4}

    \For {\label{alg:ep-bounded:step5}
    each $\boldsymbol{\ell}= (\ell_{1}, \ldots, \ell_{\omega})
   ,1 \leq \ell_{i} \leq \min (k_{i},d')$, and each $\boldpi= (
    \pi^{(1)}, \ldots, \pi^{(\omega)}) \in
    \boldpi_{\mathbf{k},\boldsymbol{\ell}}$}

     \State{ $I \gets  \{
    (i,j) \mid  1 \leq i \leq \omega,1 \leq j \leq \ell_{i}$.
    }

    \State{ Let $\ZB^{(1)}, \ldots,\ZB^{(
      \omega)}$ be new blocks of variables, where each $\ZB^{(i)}
      = (Z^{(i)}_{1}, \ldots,Z^{(i)}_{\ell_{i}})$ is a block of
      $\ell_{i}$ variables, and $Q_{\boldpi} \in \D \langle \zeta
      \rangle [ \ZB^{(1)}, \ldots,\ZB^{(\omega)} ]$ be
      the polynomial obtained from $Q$ by substituting in $Q$ for each $(i,j
     )   \in I$ the variables 
     \[
     X_{\pi^{(i)}_{1} + \cdots + \pi_{j-1}^{(i
     )} +1}, \ldots,X_{\pi_{1}^{(i)} + \cdots + \pi_{j}^{(i)}}
     \] 
     by  $Z^{(i)}_{j}$.} \label{alg:ep-bounded:step5.1}

     \State{
      \begin{eqnarray*}
        \bar{Q}_{\boldpi} & \gets & Q_{\boldpi}^{2} + \sum_{(i,j)
       , (i',j') \in I} \left(\pi^{(i)}_{j} \frac{\partial 
        Q_{\boldpi}}{\partial Z^{(i)}_{j}} - \pi^{(i')}_{j'}
        \frac{\partial Q_{\boldpi}}{\partial Z^{(i')}_{j'}}
        \right)^{2}.
      \end{eqnarray*}
      }   \label{alg:ep-bounded:step5.2}

      \State{ Using Algorithm 12.64
      (Algebraic Sampling) from \cite{BPRbook3} compute a set $\mathcal{U_{\pi}}$ of real univariate
      representations representing the finite set of points 
\[
C= \ZZ \left(
      \bar{Q}_{\boldpi}, \R \langle \zeta \rangle^{k} \right).
\]
      } \label{alg:ep-bounded:step5.3}

\algstore{myalg}
\end{algorithmic}
\end{algorithm}
 
\begin{algorithm}[H]
\begin{algorithmic}[1]
\algrestore{myalg}

      \State{
       Let $\Hess_{\boldpi} (\ZB^{(1
     )}, \ldots,\ZB^{(\omega)})$ be the symmetric matrix
      obtained by substituting for each $(i,j) \in I$ the variables
      \[X_{\pi^{(i)}_{1} + \cdots + \pi_{j-1}^{(i)} +1}, \ldots
     ,X_{\pi_{1}^{(i)} + \cdots + \pi_{j}^{(i)}},
      \] 
      by $Z^{(i)}_{j}$ in
      the $(k-1) \times (k-1)$ matrix $H$ whose rows and columns are
      indexed by $[ 2,k ]$, and which is defined by:
      \begin{eqnarray*}
        H_{i,j} & = & \dfrac{\partial^{2} Q}{\partial (X_{1} +X_{i})
        \partial (X_{1} +X_{j})}  ,2 \leq i,j \leq k.
      \end{eqnarray*} 
      }\label{alg:ep-bounded:step5.4}

      \For {  
      \label{alg:ep-bounded:step5.5}
      each point $\zb \in C$ represented by $u_{\zb} \in
      \mathcal{U}_{\boldpi}$} 
       compute using Algorithm 10.98 (Univariate
      Sign Determination) in \cite{BPRbook3}, the sign of the polynomial
      \[
      \sum_{_{(i,j) \in I}} \dfrac{\partial Q_{\boldpi}}{\partial
      Z^{(i)}_{j}},
      \] as well as the index, $\ind^{-} (
      \Hess_{\boldpi})$, at the point $\zb$.
      \EndFor

      \State{  Using Gauss-Jordan elimination (over the real
      univariate representation $u_{\zb}$), and Algorithm 10.98 (Univariate Sign Determination) from \cite{BPRbook3}, 
      determine if the negative eigenspace, $L^{-} (
      \Hess_{\boldpi} (\zb))$ of the symmetric matrix
      $\Hess_{\boldpi} (\zb)$ is contained in the subspace $L$
      defined by \[
      \sum_{i=1}^{k} X_{i} =0.
      \]
      } \label{alg:ep-bounded:step5.6}
      
      \If {$\sum_{_{(i,j) \in I}} \dfrac{\partial
      Q_{\boldpi}}{\partial Z^{(i)}_{j}}   (\zb) <0$}

      \State{\begin{eqnarray*}
        \chi^{\tmop{gen}} & \gets & \chi^{\tmop{gen}} + (-1
       )^{\ind^{-} (\Hess_{\boldpi} (\zb))}
        \prod_{i=1}^{\omega} \binom{k_{i}}{\pi^{(i)}_{1}, \ldots, \pi^{(i
       )}_{\ell_{i}}}.
      \end{eqnarray*}
      }
      \EndIf \label{alg:ep-bounded:step5.7}
      \If {$\sum_{_{(i,j) \in I}} \dfrac{\partial
      Q_{\boldpi}}{\partial Z^{(i)}_{j}}   (\zb) <0$, and $L^{-} (
      \Hess_{\boldpi} (\zb)) \subset L$}
      \State{
      \begin{eqnarray*}
        \chi^{\tmop{gen}}_{\mathfrak{S}_{\mathbf{k}}} & \gets &
        \chi_{\mathfrak{S}_{\mathbf{k}}}^{\tmop{gen}} + (-1)^{\ind^{-}
        (\Hess_{\boldpi} (\zb))}.
      \end{eqnarray*}
      } 
      \EndIf \label{alg:ep-bounded:step5.8}

    \State{ Output
    \begin{eqnarray*}
      \chi^{\tmop{gen}} \left(\ZZ \left(P, \R^{k} \right) \right)   &
      = & \chi^{\tmop{gen}},\\
      \chi^{\tmop{gen}}_{\mathfrak{S}_{\mathbf{k}}} \left(\ZZ \left(P,
      \R^{k} \right) \right) & = &
      \chi^{\tmop{gen}}_{\mathfrak{S}_{\mathbf{k}}}.
    \end{eqnarray*}
    }\label{alg:ep-bounded:step6}
  \EndFor
  \end{algorithmic}
\end{algorithm}
}

\paragraph{\sc Proof of correctness}
The correctness of the
algorithm follows from Propositions \ref{prop:alg-to-semialg},
\ref{prop:half-degree}, \ref{prop:non-degenerate}, Theorem
\ref{thm:equivariant-ep}, as well as the correctness of Algorithms 12.64 (Algebraic Sampling) and 
Algorithm 10.98 (Univariate Sign Determination) in \cite{BPRbook3}.

\paragraph{\sc Complexity analysis}
The complexity of Step \ref{alg:ep-bounded:step1} is bounded by 
$d^{O (
1)}$. The complexities of Steps 
\ref{alg:ep-bounded:step2}, \ref{alg:ep-bounded:step3}, \ref{alg:ep-bounded:step5.1}
are all bounded by $(\omega
k)^{O (d)}$. 
Using the complexity analysis of Algorithm 12.64 (Algebraic Sampling) in \cite{BPRbook3}, the complexity of Step 
\ref{alg:ep-bounded:step5.3}
is bounded by $(\length (
\boldpi) d)^{O (\length (
\boldpi))}$. The number and the degrees of the real
univariate representations output in 
Step 
\ref{alg:ep-bounded:step5.3}
are bounded by $d^{O (
\length (\boldpi))}$. The
complexity of Step 
\ref{alg:ep-bounded:step5.5}
 is bounded by $d^{O (
\length (\boldpi))}$ using
the complexity analysis of Algorithm 10.98 (Univariate Sign Determination) in \cite{BPRbook3}. Each arithmetic operation in the Gauss-Jordan elimination in
Step 
\ref{alg:ep-bounded:step5.6}
occurs in a ring $D [ \zeta ] [ T ] / (f (T))$ (where $u_{\zb} = (
f,g_{0,} \ldots,g_{\length (
\boldpi)}), \rho_{\zb}$) with $\deg_{T, \zeta} (f)
=d^{O (\length (\boldpi)
)}$). The number of such operations in the ring $D [ \zeta ] [ T ] / (f (T)
)$ is bounded by $(\length (
\boldpi) +k)^{O (1)}$. Thus, the total number of
arithmetic operations in the ring $\D$ performed in Step 
\ref{alg:ep-bounded:step5.6}
  is bounded by
$(\length (\boldpi) k d
)^{O (\length (\boldpi)
)}$.

  The number of iterations of 
Step  \ref{alg:ep-bounded:step5} is bounded by the number of partitions
$\boldpi \in \Pi_{\mathbf{k},\boldsymbol{\ell}}$ with
$\boldsymbol{\ell}= (\ell_{1}, \ldots, \ell_{\omega}),1 \leq \ell_{i} \leq \min
(k_{i},d'),1 \leq i \leq \omega$, which is bounded by
\begin{eqnarray*}
  \sum_{\boldsymbol{\ell}= (\ell_{1}, \ldots, \ell_{\omega}),1 \leq \ell_{i}
  \leq \min (k_{i},d')} p (\mathbf{k},\boldsymbol{\ell}) & = & k^{O (D)},
\end{eqnarray*}
where $D = \sum_{i=1}^{\omega} \min(k_i,2d)$.
Thus, the total complexity of the algorithm measured by the number of
arithmetic operations (including comparisons) in the ring $\D$ is bounded by
$(\omega  k d)^{O (D)}$.

{\small
\begin{algorithm}[H]
\caption{(Generalized Euler-Poincar{\'e}
characteristic for symmetric algebraic sets)}
\label{alg:ep-general-BM}
\begin{algorithmic}[1]
\INPUT
\Statex{

  \begin{enumerate}
    \item A tuple $\mathbf{k}= (k_{1}, \ldots,k_{\omega}) \in
    \Z_{>0}^{\omega}$, with $k= \sum_{i=1}^{\omega} k_{i}$.
    
    \item A polynomial $P \in \D [ \X^{(1)}, \ldots
   ,\X^{(\omega)} ]$, where each $\X^{(i)}$ is a
    block of $k_{i}$ variables, and $P$ is symmetric in each block of
    variables $\X^{(i)}$, with $\deg (P) =d$.
  \end{enumerate}
  }
 \OUTPUT
 \Statex{
   $\chi^{\tmop{gen}} \left(\ZZ \left(P, \R^{k} \right) \right)  $
  and 
  $\chi_{\mathfrak{S}_{\mathbf{k}}}^{\tmop{gen}} \left(\ZZ \left(P, \R^{k}
  \right) \right)$.
  }

\PROCEDURE
  
\State{
    \begin{eqnarray*}
      P_{1} & \gets & P^{2} + \left(X_{k+1}^{2} + \sum_{i=1}^{k} X_{i}^{2} + 
      - \Omega^{2} \right)^{2}  ,\\
      P_{2} & = & P^{2} + \left(\sum_{i=1}^{k} X_{i}^{2}  - \Omega^{2}
      \right)^{2}.
    \end{eqnarray*}
    }\label{alg:ep-general-BM:step1}

\State{
    Using Algorithm \ref{alg:ep-bounded} with $P_{1}$ and
    $P_{2}$ as input compute:
    \begin{eqnarray*}
      \chi^{(1)} & = & \chi^{\tmop{gen}} \left(\ZZ \left(P_{1}, \R
      \langle 1/ \Omega \rangle^{k+1} \right) \right),\\
      \chi^{(1)}_{\mathfrak{S}_{\mathbf{k}'}} & = &
      \chi_{\mathfrak{S}_{\mathbf{k}'}}^{\tmop{gen}} \left(\ZZ \left(P_{1},
      \R \langle 1/ \Omega \rangle^{k+1} \right) \right),\\
      \chi^{(2)} & = & \chi^{\tmop{gen}} \left(\ZZ \left(P_{2}, \R
      \langle 1/ \Omega \rangle^{k} \right) \right),\\
      \chi^{(2)}_{\mathfrak{S}_{\mathbf{k}}} & = &
      \chi_{\mathfrak{S}_{\mathbf{k}}}^{\tmop{gen}} \left(\ZZ \left(P_{2},
      \R \langle 1/ \Omega \rangle^{k} \right) \right),
    \end{eqnarray*}
    where $\mathbf{k}' = (\mathbf{k},1)$.
    }\label{alg:ep-general-BM:step2}

\State{  Output
    \begin{eqnarray*}
      \chi^{\tmop{gen}} \left(\ZZ \left(P, \R^{k} \right) \right) & = &
      \tfrac{1}{2} (\chi^{(1)} - \chi^{(2)}),\\
      \chi_{\mathfrak{S}_{\mathbf{k}}}^{\tmop{gen}} \left(\ZZ \left(P,
      \R^{k} \right) \right) & = & \tfrac{1}{2} (
      \chi_{\mathfrak{S}_{\mathbf{k}'}}^{(1)} -
      \chi_{\mathfrak{S}_{\mathbf{k}}}^{(2)}).
    \end{eqnarray*}
    }\label{alg:ep-general-BM:step3}
 \end{algorithmic}
 \end{algorithm}
}

\paragraph{\sc Proof of correctness} 
Since $V= \ZZ \left(P,
\R^{m+k} \right)$ is closed, by Proposition \ref{prop:ep-unbounded} we have
that
\begin{eqnarray}
\nonumber
  \chi^{\tmop{gen}} (V) & = & \chi^{\tmop{BM}} (V,\mathbb{Q}) \\
  \nonumber
  & = & \chi^{\mathrm{top}} \left(\Ext \left(V, \R \langle 1/ \Omega \rangle \right) \cap
  \overline{B_{k} (0, \Omega)} \right) - \\
  \nonumber
  &&\chi^{\mathrm{top}} \left(\Ext \left(V, \R
  \langle 1/ \Omega \rangle \right) \cap S^{k-1} (0, \Omega) \right)
  \nonumber\\
  & = & \chi^{\mathrm{top}} \left(\Ext \left(V, \R \langle 1/ \Omega \rangle \right) \cap
  \overline{B_{k} (0, \Omega)} \right) - \chi^{(2)}.  \label{eqn:proof1}
\end{eqnarray}
Now $\ZZ \left(P_{1}, \R \langle 1/ \Omega \rangle^{k+1} \right)$ is
semi-algebraically homeomorphic to two copies of 
$$\Ext \left(V, \R \langle 1/
\Omega \rangle \right) \cap \overline{B_{k} (0, \Omega)},
$$ 
glued along a
semi-algebraically homeomorphic copy of 
$$
\Ext \left(V, \R \langle 1/ \Omega
\rangle \right) \cap S^{k-1} (0, \Omega) = \ZZ \left(P_{2}, \R \langle 1/
\Omega \rangle^{k} \right).
$$ 
It follows that,
\begin{eqnarray*}
   \chi^{(1)} &=&\chi^{\tmop{gen}} \left(\ZZ \left(P_{1}, \R \langle 1/ \Omega
  \rangle^{k+1} \right) \right) \\
  & = & 2 \chi^{\mathrm{top}} \left(\Ext \left(V, \R \langle 1/ \Omega \rangle \right)
  \cap \overline{B_{k} (0, \Omega)} \right)  - \chi^{\tmop{gen}} \left(\ZZ
  \left(P_{2}, \R \langle 1/ \Omega \rangle^{k} \right) \right)\\
  & = & 2 \chi^{\mathrm{top}} \left(\Ext \left(V, \R \langle 1/ \Omega \rangle \right)
  \cap \overline{B_{k} (0, \Omega)} \right) - \chi^{(2)},
\end{eqnarray*}
and hence
\begin{eqnarray}
  \chi^{\mathrm{top}} \left(\Ext \left(V, \R \langle 1/ \Omega \rangle \right) \cap
  \overline{B_{k} (0, \Omega)} \right) & = & \tfrac{1}{2} (\chi^{(1)} +
  \chi^{(2)}).  \label{eqn:proof2}
\end{eqnarray}
It now follows from Eqns. (\ref{eqn:proof1}) and (\ref{eqn:proof2}) that
\begin{eqnarray*}
  \chi^{\tmop{gen}} (V) & = & \tfrac{1}{2} (\chi^{(1)} + \chi^{(2)}) -
  \chi^{(2)}\\
  & = & \tfrac{1}{2} (\chi^{(1)} - \chi^{(2)}).
\end{eqnarray*}
The proof for the correctness of the computation of
$\chi^{\tmop{gen}}_{\mathfrak{S}_{\mathbf{k}}} (V)$ is similar and
omitted.

\paragraph{\sc Complexity analysis}
The complexity of the algorithm measured by the
number of arithmetic operations (including comparisons) in the ring $\D$ is
bounded by $(\omega  k d)^{O (D)}$, where $D = \sum_{i=1}^\omega \min(k_i,2d)$. This follows
directly from the complexity analysis of Algorithm \ref{alg:ep-bounded}.

\begin{proof}[Proof of Theorem \ref{thm:algorithm-algebraic}] The
correctness and the complexity analysis of Algorithm \ref{alg:ep-general-BM}
prove Theorem \ref{thm:algorithm-algebraic}.
\end{proof}

\subsection{Computing the generalized Euler-Poincar{\'e} characteristic of
symmetric semi-algebraic sets}

We now consider the problem of computing the (generalized)
Euler-Poincar{\'e} characteristic of semi-algebraic sets. We reduce the
problem to computing the generalized Euler-Poincar{\'e} characteristic of
certain symmetric algebraic sets for which we already have an efficient
algorithm described in the last section. This reduction process follows very
closely the spirit of a similar reduction that is used in an algorithm for
computing the generalized Euler-Poincar{\'e} characteristic of the
realizations of all realizable sign conditions of a family of polynomials
given in {\cite{BPR-euler-poincare}} (see also \cite{BPRbook3}).

We first need an efficient algorithm for computing the set of realizable sign
conditions of a family of symmetric polynomials which will be used later. The
following algorithm can be considered as an equivariant version of a very
similar algorithm -- namely, Algorithm 13.9 (Computing Realizable Sign
Conditions) in \cite{BPRbook3} -- for solving the same problem in the
non-equivariant case.

{\small
\begin{algorithm}[H]
\caption{(Computing Realizable Sign Conditions of Symmetric
Polynomials)}
\label{alg:sampling}
\begin{algorithmic}[1]
\INPUT
\Statex{
\begin{enumerate}
    \item A tuple $\mathbf{k}= (k_{1}, \ldots,k_{\omega}) \in
    \Z_{>0}^{\omega}$, with $k= \sum_{i=1}^{\omega} k_{i}$.
    
    \item A set of $s$ polynomials $\mathcal{P}= \{ P_{1}, \ldots,P_{s} \}
    \subset \D [ \X^{(1)}, \ldots,\X^{(\omega)} ]$,
    where each $\X^{(i)}$ is a block of $k_{i}$ variables, and
    each polynomial in $\mathcal{P}$ is symmetric in each block of variables
    $\X^{(i)}$ and of degree at most $d$.
  \end{enumerate}
}
\OUTPUT
\Statex{
  $\SIGN (\mathcal{P})$.
  }
  
 \PROCEDURE

    \For {each $i,1 \leq i \leq s$   \label{alg:sampling:step1}}  
        $\mathcal{P}^{\star}_{i} \gets  \{ P_{i} \pm \gamma \delta_{i},P_{i} \pm \delta_{i} \}$.
    \EndFor
    
    \For {
    \label{alg:sampling:step2}
    every choice of $j \le D'  = 
    \sum_{i=1}^{\omega} \min (k_{i},d)$ polynomials $Q_{i_{1}} \in
    \mathcal{P}_{i_{1}}^{\star}, \ldots,Q_{i_{j}} \in
    \mathcal{P}_{i_{j}}^{\star}$, with $1 \leq i_{1} < \cdots <i_{j} \leq s$
    }
    
    \State{
      
      $$\displaylines{ 
      Q_1\gets Q_{i_{1}}^{2} + \cdots +Q_{i_{j}}^{2}, \cr
      Q_2 \gets Q_{i_{1}}^{2} + \cdots +Q_{i_{j}}^{2} + \left(\eps(|\X^{(1)}|^2+\cdots+|\X^{(\omega)}|^2) -1\right)^2.
      }
      $$
      \label{alg:sampling:step2.1}
      }

      \For {each $\boldsymbol{\ell}= (\ell_{1}, \ldots,
      \ell_{\omega}),1 \leq \ell_{i} \leq \min (k_{i},4d)$, and each
      partition $\boldpi= (\pi^{(1)}, \ldots, \pi^{(\omega)})
      \in \boldpi_{\mathbf{k},\boldsymbol{\ell}}$
      \label{alg:sampling:step2.2}
      }

       \State{ $I
      \gets  \{ (i,j) \mid  1 \leq i \leq \omega,1 \leq j \leq \ell_{i} \}
       $
       \label{alg:sampling:step2.2.0}
       }

      \State{ 
        Let $\ZB^{(1)}, \ldots,\ZB^{(
        \omega)}$ be new blocks of variables, where each $\ZB^{(i
       )} = (Z^{(i)}_{1}, \ldots,Z^{(i)}_{\ell_{i}})$ is a block of
        $\ell_{i}$ variables, and 
$
Q_{1,\boldpi},  
Q_{2,\boldpi} 
$ be the
        polynomials obtained from $Q_1,Q_2$ respectively, by substituting 
        for each $(i,j) \in I$ the variables \[
        X_{\pi^{(i)}_{1} + \cdots + \pi_{j-1}^{(i)}
        +1}, \ldots,X_{\pi_{1}^{(i)} + \cdots + \pi_{j}^{(i)}}
        \] by $Z^{(
        i)}_{j}$.
        \label{alg:sampling:step2.2.1}
        }

        \State{ Using Algorithm 12.64 
        (Algebraic Sampling) from \cite{BPRbook3},
         compute a set $\mathcal{U}_{\boldpi,i}, i = 1,2$ of
        real univariate representations representing the finite set of points
        \[
        C_1, C_2 \subset \ZZ \left(Q_{1,\boldpi}, \R'^{\length (
        \boldpi)} \right),
        \] 
        where $\R' = \R \langle \eps, \delta_{i_{1}},
        \ldots, \delta_{i_{j}}, \gamma,\zeta\rangle$.
        \label{alg:sampling:step2.2.2}
        }
        \State{Apply the $\lim_{\gamma}$ using  Algorithm 12.46 (Limit of Bounded Points) in \cite{BPRbook3},  to the points in $C_1,C_2$ which are bounded over 
        $ \R \langle \eps,\delta_{i_{1}},
        \ldots, \delta_{i_{j}} \rangle$,  and obtain a set of real
      univariate representations $(u, \sigma )$ with
      \[ u= (f(T),g_{0} (T), \ldots,g_{\length(\boldpi)} (T)) \in \D [\eps,\delta_{i_1},\ldots,\delta_{i_j}] [T]^{\length(\boldpi)+2} \]
      Add these real univariate representations to $\mathcal{U}_{\boldpi}$.
      \label{alg:sampling:step2.2.3}}
    \For {each $u \in \mathcal{U}_{\boldpi}$    \label{alg:sampling:step2.2.4}}
        \State{Compute the signs of $P_{\boldpi}$ for each $P \in \mathcal{P}$ at the points $\zb$, associated to $u$ using Algorithm 10.98
    (Univariate Sign Determination) from \cite{BPRbook3},
    where $P_{\boldpi}
        \in \R' [ \ZB^{(1)}, \ldots,\ZB^{(\omega)} ]$
        is the polynomial obtained from $P$ by substituting in $P$ for each $(
        i,j)   \in I$ the variables 
        \[X_{\pi^{(i)}_{1} + \cdots +
        \pi_{j-1}^{(i)} +1}, \ldots,X_{\pi_{1}^{(i)} + \cdots +
        \pi_{j}^{(i)}}\]
         by $Z^{(i)}_{j}$.
    \label{alg:sampling:step2.2.5}
    }

\algstore{myalg}
\end{algorithmic}
\end{algorithm}
 
\begin{algorithm}[H]
\begin{algorithmic}[1]
\algrestore{myalg}

    \State{ Let $\sigma_{\zb} \in \{0,1,-1\}^{\mathcal{P}}$ be the sign vector defined by $\sigma(P) = \sign(P_{\boldpi}(\zb))$.
    \label{alg:sampling:step2.2.6}
    }
      \State{ $\tmop{SIGN}  :=  \tmop{SIGN}   \cup \{ \sigma_{\zb} \}$. \label{alg:sampling:step2.2.7}
      }
      \EndFor
        
        \EndFor
    \EndFor

    \State{Output $\tmop{SIGN} (\mathcal{P}) = \tmop{SIGN}$.} \label{alg:sampling:step3}
\end{algorithmic}
\end{algorithm}
}

\paragraph{\sc Proof of correctness}

We first need a lemma whose proof can be found in \cite{BC2013}.
\begin{definition}
  For any finite family $\mathcal{P} \subset \R [ X_{1}, \ldots,X_{k} ]$ and
  $\ell \geq 0$, we say that $\mathcal{P}$ is in $\ell$-general position with
  respect to a semi-algebraic set $V \subset \R^{k}$ if for any subset
  $\mathcal{P}' \subset \mathcal{P}$, with $\card (
  \mathcal{P}') > \ell$, $\ZZ (\mathcal{P}',V) = \emptyset$. 
\end{definition}

Let $\mathbf{k}= (k_{1}, \ldots,k_{\omega})$ with $k=
\sum_{i=1}^{\omega} k_{i}$, and 
\[
\mathcal{P} = \{ P_{1}, \ldots,P_{s} \} \subset \R [\X^{(1)}, \ldots,\X^{(\omega)} ]^{\mathfrak{S}_\kk}
\]
be a fixed finite set of polynomials where $\X^{(i)}$ is a block of $k_{i}$
variables.
Let $\deg (P_{i}) \leq d$ for $1 \leq
i \leq s$. Let $\overline{\eps} = \left(\eps_{1}, \ldots, \eps_{s} \right)$
be a tuple of new variables, and let $\mathcal{P}_{\overline{\eps}} =
\bigcup_{1 \leq i \leq s} \left\{ P_{i}   \pm \eps_{i} \right\}$. 

The following lemma appears in \cite{BC2013}.

\begin{lemma}\cite[Lemma 7]{BC2013}
  \label{lem:gen-pos1-with-parameters}Let
  \begin{eqnarray*}
    D'  & = & \sum_{i=1}^{\omega} \min (k_{i},d).
  \end{eqnarray*}
  The set of polynomials  $\mathcal{P}_{\overline{\eps}} \subset \R' [
  \X^{(1)}, \ldots,\X^{(
  \omega)} ]$ is in $D'$-general position for any semi-algebraic subset $Z
  \subset \R^{k}$ stable under the action of $\mathfrak{S}_{\mathbf{k}}$,
  where $\R' = \R \langle \overline{\eps} \rangle$.
\end{lemma}

Now observe that Lemma \ref{lem:gen-pos1-with-parameters} implies
that the set $\bigcup_{1 \leq i \leq s} \mathcal{P}_{i}^{\star}$ is in $D'$-general position. Propositions 13.1 and 13.7 in \cite{BPRbook3} together imply  that the image under the $\lim_{\gamma}$ map of any finite set of
points meeting every bounded semi-algebraically connected component of each algebraic
set defined by polynomials 
$$\displaylines{
Q_{i_{1}} \in \mathcal{P}_{i_{1}}^{\star}, \ldots,Q_{i_{j}} \in \mathcal{P}_{i_{j}}^{\star},\cr
Q_{i_{1}} \in \mathcal{P}_{i_{1}}^{\star}, \ldots,Q_{i_{j}} \in \mathcal{P}_{i_{j}}^{\star}, Q_0,
}
$$
where  $1  \leq i_{1} < \cdots <i_{j}  
\leq s$, $1 \leq j \leq D'$, and
$Q_0 = |\eps|(|\X^{(1)}|^2+ \cdots + |\X^{(\omega)}|^2)-1$, will intersect every
semi-algebraically connected component of $\RR \left(\sigma, \R^{k} \right)$
for every $\sigma \in \tmop{SIGN} (\mathcal{P})$. 

Moreover, noticing that the degrees of the polynomials $Q_{i_j}$ above are bounded by $2d$, it follows from Corollary \ref{cor:half-degree} that each semi-algebraically connected component of the algebraic sets listed above has a non-empty intersection with
$A^{\boldsymbol{\ell}}_{\kk}$, for some $\boldsymbol{\ell}= (\ell_1,\ldots,\ell_\omega)$,
and $1 \leq \ell_1 \leq \min(k_i,4d), 1\leq i \leq \omega$.

The correctness of the
algorithm now follows from the correctness of Algorithm 12.64
(Algebraic Sampling)  and
Algorithm 10.98 (Univariate Sign Determination) in \cite{BPRbook3}.

\paragraph{\sc Complexity analysis}
The complexity of Step 
\ref{alg:sampling:step2}
  measured by the number of arithmetic operations in
the ring $\D [ \delta_{1}, \ldots, \delta_{s}, \gamma ]$ is bounded by 
\[
O
\left(D'\binom{k+d}{k} \right),
\] 
where $D' = \sum_{i=1}^{\omega} \min(k_i,d)$.

It follows from the
complexity analysis of Algorithm 12.64 (Algebraic Sampling) in \cite{BPRbook3}  that each call to Algorithm 
12.64  (Algebraic Sampling) in Step 
\ref{alg:sampling:step2.2.2}
  requires $d^{O (\length (\boldpi))}$
arithmetic operations in the ring $\D [\eps, \delta_{1}, \ldots, \delta_{s},
\gamma ]$. The number and degrees of the real univariate representations $u_{\boldpi,i}$ output in Step 
\ref{alg:sampling:step2.2.2}
  is bounded by $d^{O (\length(\boldpi))}$.
Using the complexity analysis of Algorithm 12.46  (Limit  of Bounded Points) in \cite{BPRbook3}, each call to Algorithm 12.46   (Limit of Bounded Points)  in Step
\ref{alg:sampling:step2.2.3}
  requires $d^{O (\length(
\boldpi))}$ arithmetic operations in the ring $\D[\eps, \delta_1,\ldots,\delta_s,\gamma]$,
and thus the total complexity of this step
in the whole algorithm across all iterations 
measured by the number of arithmetic operations in the ring $\D
[ \eps,\delta_{1}, \ldots, \delta_{s}]$ is bounded by 
$$\sum_{j=1}^{D'} 2^{j} \binom{s}{j} \left(d^{O (D'')} +O
\left(D'  \binom{k+d}{k} \right) \right),$$
where $D'' = \sum_{i=1}^\omega \min(k_i,4d)$,
noting that
$\length (\boldpi) \leq D''$.

Similarly, using the complexity analysis of Algorithm 10.98  (Univariate Sign Determination) in \cite{BPRbook3}, each call to Algorithm 
10.98 (Univariate Sign Determination) in Step
\ref{alg:sampling:step2.2.5}
  requires $d^{O (\length(
\boldpi))}$ arithmetic operations in the ring $\D[\eps, \delta_1,\ldots,\delta_s]$,
and thus the total complexity of this step
in the whole algorithm across all iterations
measured by the number of arithmetic operations in the ring $\D
[ \eps,\delta_{1}, \ldots, \delta_{s}]$ is bounded by 

$$\sum_{j=1}^{D'} 2^{j} \binom{s}{j} \left(d^{O (D'')} +O
\left(D'  \binom{k+d}{k} \right) \right).$$
However, notice
that in each call to  Algorithm 12.64 (Algebraic
Sampling) from \cite{BPRbook3} in Step 
\ref{alg:sampling:step2.2.2},
to Algorithm 12.46 (Limits of Bounded Points) in \cite{BPRbook3} in Step \ref{alg:sampling:step2.2.3}, as well as
and also in the calls to Algorithm 10.98 (Univariate Sign Determination) from \cite{BPRbook3} in Step 
\ref{alg:sampling:step2.2.5}, the
arithmetic is done in a ring $\D$ adjoined with $O (D')$
infinitesimals. Hence, the total number of arithmetic operations in $\D$ is
bounded by

$$\displaylines{
\sum_{j=1}^{D'} 2^{j} \binom{s}{j} \left(d^{O (
D'D'' )} +O \left(D'   \binom{k+d}{k}
\right) \right) = 
s^{D'} k^{d} d^{O (
D'D'')}.
}
$$

The total number of real univariate representations produced in Step 
\ref{alg:sampling:step2.2.2}
  is
bounded by 

$$
\sum_{j=1}^{D'} 2^{j} \binom{s}{j} d^{O (D'')} =s^{D'} d^{O (D'')}.
$$
Their degrees are bounded by $d^{O (D'')}$. Thus, the total
number of real points associated to these univariate representations, and
hence also 
\[
\card (\tmop{SIGN} (\mathcal{P}))  = s^{D'} d^{O (D'')}.
\]

The complexity analysis of Algorithm \ref{alg:sampling} yields the following
purely mathematical result.

\begin{proposition}
  \label{prop:SIGN} Let $\mathbf{k}= (k_{1}, \ldots,k_{\omega}) \in
  \Z_{>0}^{\omega}$, with $k= \sum_{i=1}^{\omega} k_{i}$, and let
  $\mathcal{P}= \{ P_{1}, \ldots,P_{s} \} \subset \R [
  \X^{(1)}, \ldots,\X^{(
  \omega)} ]$ be a finite set of polynomials, where each
  $\X^{(i)}$ is a block of $k_{i}$ variables, and
  each polynomial in $\mathcal{P}$ is symmetric in each block of variables
  $\X^{(i)}$. Let $\card
  (\mathcal{P}) =s$, and $\max_{P \in \mathcal{P}}   \deg (P) =d$. Then,
  \begin{eqnarray*}
    \card ( \SIGN (
    \mathcal{P})) & = & s^{D'} d^{O (D''
   )} ,
  \end{eqnarray*}
  where $D'= \sum_{i=1}^{\omega} \min (k_{i},d)$, and 
  $D''= \sum_{i=1}^{\omega} \min (k_{i},4d)$.
  
  In particular, if for each $i,1 \leq i \leq \omega $, $d \leq k_{i}$, then
  $\card ( \SIGN (
  \mathcal{P}))$ can be bounded independent of $k$.
\end{proposition}

\begin{notation}
  Given $P \in \R [X_{1}, \ldots,X_{k} ]$, we denote
  \begin{eqnarray*}
    \RR (P=0,S) & = & \{x \in S \hspace{0.75em} \mid \hspace{0.75em} P(x)=0\}
   ,\\
    \RR (P>0,S) & = & \{x \in S \hspace{0.75em} \mid \hspace{0.75em} P(x)>0\}
   ,\\
    \RR (P>0,S) & = & \{x \in S \hspace{0.75em} \mid \hspace{0.75em} P(x)<0\}
   ,
  \end{eqnarray*}
  and $\chiep^{\gen} (P=0,S),
  \chiep^{\gen} (P>0,S),
  \chiep^{\gen} (P<0,S)$ the Euler-Poincar{\'e}
  characteristics of the corresponding sets. The Euler-Poincar{\'e}-query of
  $P$ for $S$ is
  \[ \EQ (P,S) = \chiep^{\gen} (P>0,S) -
     \chiep^{\gen} (P<0,S). \]
  If $P$ and $S$ are symmetric we denote by
  \[\chiep^{\gen}_{\mathfrak{S}_{k}} (P=0,S),
  \chiep^{\gen}_{\mathfrak{S}_{k}} (P>0,S),
  \chiep^{\gen}_{\mathfrak{S}_{k}} (P<0,S)
  \] 
  the
  Euler-Poincar{\'e} characteristics of the corresponding sets. The
  equivariant Euler-Poincar{\'e}-query of $P$ for $S$ is
  \[ \EQ_{\mathfrak{S}_{k}} (P,S) =
     \chiep^{\gen}_{\mathfrak{S}_{k}} (P>0,S) -
     \chiep^{\gen}_{\mathfrak{S}_{k}} (P<0,S). \]
  Let $\mathcal{P} =P_{1}, \ldots,P_{s}$ be a finite list of polynomials in
  $\R [X_{1}, \ldots,X_{k} ]$.
  
  \label{13:def:realization signcondition2}Let $\sigma$ be a sign condition on
  $\mathcal{P}$. The \emph{realization of the sign condition $\sigma$ over
  $S$} is defined by
  \[ \RR (\sigma,S) = \{x \in S  \mid 
     \bigwedge_{P \in \mathcal{P}} \sign (P(x))= \sigma (P)\}, \]
  and its generalized Euler-Poincar{\'e} characteristic is denoted
  \[
  \chiep^{\gen} (\sigma,S).
  \] 
  Similarly, if $P$
  and $S$ are symmetric with respect to $\mathfrak{S}_{\mathbf{k}}$ for some
  $\mathbf{k}= (k_{1}, \ldots,k_{\omega}) \in \Z_{>0}^{\omega}$,
  the equivariant Euler-Poincar{\'e} characteristic of $\RR (\sigma,S)$ is
  denoted
  \[\chiep^{\gen}_{\mathfrak{S}_{\mathbf{k}}} (
  \sigma,S) :=
  \chi^{\gen}_{\mathfrak{S}_{\mathbf{k}}} \left(
  \phi_{\mathbf{k}} \left(\RR (\sigma,S) \right),\mathbb{Q} \right).
  \]
\end{notation}

\begin{notation}
  \label{13:not:chichibar}\label{13:not:signdet}Given a finite family
  $\mathcal{P} \subset \R [ X_{1}, \ldots,X_{k} ]$ we denote by
  $\chiep^{\gen} (\mathcal{P})$ the list of
  generalized Euler-Poincar{\'e}
  characteristics \[
  \chiep^{\gen} (\sigma) =
  \chiep^{\gen} (\RR (\sigma, \R^{k}))
  \] 
  for $\sigma \in \SIGN (\mathcal{P})$. 
\end{notation}

Given $\alpha \in \{0,1,2\}^{\mathcal{P}}$ and $\sigma \in \{0,1, - 1\}^{\mathcal{P}}$, we denote  
\[
\sigma^{\alpha} = \prod_{P \in \mathcal{P}} \sigma (P)^{\alpha (P)},
\] 
and
\[
\mathcal{P}^{\alpha}= \prod_{P \in \mathcal{P}} P^{\alpha (P)}.
\]

When $\RR (\sigma,Z) \ne
\emptyset$, the sign of $\mathcal{P}^{\alpha}$ is fixed on $\RR (\sigma,Z)$
and is equal to $\sigma^{\alpha}$ with the understanding that $0^{0} =1$.

We order the elements of $\mathcal{P}$ so that $\mathcal{P} = \{P_{1}, \ldots
,P_{s} \}$. We order $\{0,1,2\}^{\mathcal{P}}$ lexicographically. We also
order $\{0,1, - 1\}^{\mathcal{P}}$ lexicographically (with $0 \prec 1 \prec
-1$).

Given $A= \alpha_{1}, \ldots, \alpha_{m}$, a list of elements of
$\{0,1,2\}^{\mathcal{P}}$ with $\alpha_{1} <_{\lex} \ldots <_{\lex}
\alpha_{m}$, we define
\begin{eqnarray*}
  \mathcal{P}^{A} & =  & \mathcal{P}^{\alpha_{1}}, \ldots,
  \mathcal{P}^{\alpha_{m}},\\
   \EuQ (\mathcal{P}^{A},S) & = &
  { \EuQ (\mathcal{P}^{\alpha_{1}},S),
  \ldots, \EuQ (\mathcal{P}^{\alpha_{m}},S).}
\end{eqnarray*}
Given $\Sigma = \sigma_{1}, \ldots, \sigma_{n}$, a list of elements of
$\{0,1, - 1\}^{\mathcal{P}}$, with $\sigma_{1} <_{\lex} \ldots <_{\lex}
\sigma_{n}$,we define
\begin{eqnarray*}
  \RR (\Sigma,S) & = & \RR (\sigma_{1},Z), \ldots, \RR (\sigma_{n},Z)
 ,\\
  \chi^{\gen} (\Sigma,S) & = &
  \chi^{\gen} (\sigma_{1},Z), \ldots,
  \chi^{\gen} (\sigma_{n},Z).
\end{eqnarray*}

We denote by $\ensuremath{\operatorname{Mat}} (A, \Sigma)$ the $m \times s$
matrix of signs of $\mathcal{P}^{A}$ on $\Sigma$ defined by
\begin{eqnarray*}
  \ensuremath{\operatorname{Mat}} (A, \Sigma)_{i,j} & = &
  \sigma_{j}^{\alpha_{i}}.
\end{eqnarray*}
\begin{proposition}
  \label{13:prop:matrix of signs}If $\cup_{\sigma \in \Sigma} \RR (\sigma,S)
  =S$, then
  \[ \ensuremath{\operatorname{Mat}} (A, \Sigma) \cdot
     \chiep^{\gen} (\Sigma,S) = \EQ (
     \mathcal{P}^{A},S). \]
\end{proposition}

\begin{proof}See \cite[Proposition 13.44]{BPRbook3}.\end{proof}

We consider a list $A$ of elements in $\{0,1,2\}^{\mathcal{P}}$ \emph{adapted
to sign determination} for $\mathcal{P}$ (cf. \cite[Definition 10.72]{BPRbook3}), 
i.e. such that the matrix of signs of $\mathcal{P}^{A}$ over $\SIGN (
\mathcal{P})$ is invertible. If {$\mathcal{P} =P_{1}, \ldots,P_{s}$},
let~${\mathcal{P}_{i} =P_{i}, \ldots,P_{s}}$, for $0 \le i \le s$. A method
for determining a list $A (\mathcal{P})$ of elements in
$\{0,1,2\}^{\mathcal{P}}$ adapted to sign determination for $\mathcal{P}$ \
from $\SIGN (\mathcal{P})$ is given in Algorithm 10.83 (Adapted Matrix) in \cite{BPRbook3}.

We are ready for describing the algorithm computing the generalized
Euler-Poincar{\'e} characteristic. We start with an algorithm for the
Euler-Poincar{\'e}-query.

{\small
\begin{algorithm}[H]
\caption{(Euler-Poincar{\'e}-query)}
\label{12:alg:speuler}
\begin{algorithmic}[1]
\INPUT
\Statex{
\begin{enumerate}
    \item A tuple $\mathbf{k}= (k_{1}, \ldots,k_{\omega}) \in
    \Z_{>0}^{\omega}$, with $k= \sum_{i=1}^{\omega} k_{i}$.
    
    \item Polynomials $P,Q \in \D [ \X^{(1)}, \ldots
   ,\X^{(\omega)} ]$, where each $\X^{(i)}$ is a
    block of $k_{i}$ variables, and $P,Q$ are symmetric in each block of
    variables $\X^{(i)}$, and of degree at most $d$. 
  \end{enumerate}
}  

 \OUTPUT
 \Statex{ 
  The Euler-Poincar{\'e}-queries
  \begin{eqnarray*}
    \EQ (P,Z) & = & \chiep^{\gen} (P>0,Z) -
    \chiep^{\gen} (P<0,Z),\\
    \EQ_{\mathfrak{S}_{k}} (P,Z) & = &
    \chiep^{\gen}_{\mathfrak{S}_{k}} (P>0,Z) -
    \chiep^{\gen}_{\mathfrak{S}_{k}} (P<0,Z) ,
  \end{eqnarray*}
  where $Z=\ZZ\left(Q, \R^{k} \right)$.
}

\PROCEDURE
  
    \State{ Introduce a new variable $X_{k+1}$, and let
    \begin{eqnarray*}
      Q_{+} & = & Q^{2} + (P-X_{k+1}^{2})^{2},\\
      Q_{-} & = & Q^{2} + (P+X_{k+1}^{2})^{2}.
    \end{eqnarray*}
    } \label{12:alg:speuler:step1}
    
    \State{ Using Algorithm \ref{alg:ep-general-BM} compute
    \[ \chiep^{\tmop{gen}} (\ZZ (Q_{+}, \R^{k+1})),
    \chiep^{\tmop{gen}}_{\mathfrak{S}_{k}} (\ZZ (Q_{+}, \R^{k+1})),
    \] and
    \[\chiep^{\tmop{gen}} (\ZZ (Q_{-}, \R^{k+1})),
    \chiep^{\tmop{gen}}_{\mathfrak{S}_{k}} (\ZZ (Q_{-}, \R^{k+1})).
    \]
    } \label{12:alg:speuler:step2}
    
    \State{Output
    
    \[ (\chiep^{\tmop{gen}} (\ZZ (Q_{+}, \R^{k+1}))- \chiep^{\tmop{gen}} (
       \ZZ (Q_{-}, \R^{k+1}))) /2, \]
    \[ (\chiep^{\tmop{gen}}_{\mathfrak{S}_{k}} (\ZZ (Q_{+}, \R^{k+1}))-
       \chiep^{\tmop{gen}}_{\mathfrak{S}_{k}} (\ZZ (Q_{-}, \R^{k+1}))) /2.
    \]
    }  \label{12:alg:speuler:step3}
  \end{algorithmic}
  \end{algorithm}
}

\vspace{-0.5in}
\paragraph{\sc Proof of correctness}
The algebraic set 
$\ZZ (Q_{+}, \R^{k+1})$ 
is semi-algebraically homeomorphic
to the disjoint union of two copies of the semi-algebraic set defined by
$(P>0) \wedge (Q=0)$, and the algebraic set defined by $(P=0) \wedge (Q=0)$.
Hence, using Corollary \ref{cor:additive}, we have that
\begin{eqnarray*}
  2 \chiep^{\tmop{gen}} (P>0,Z)  &=&  \chiep^{\tmop{gen}} (\ZZ (Q_{+},
  \R^{k+1})) - \chiep^{\tmop{gen}} (\ZZ ((Q,P), \R^{k})),\\
  2 \chiep^{\tmop{gen}}_{\mathfrak{S}_{k}} (P>0,Z)  &=& 
  \chiep^{\tmop{gen}}_{\mathfrak{S}_{k}} (\ZZ (Q_{+}, \R^{k+1})) -
  \chiep^{\tmop{gen}}_{\mathfrak{S}_{k}} (\ZZ ((Q,P), \R^{k})).
\end{eqnarray*}

Similarly, we have that
\begin{eqnarray*}
  2  \chiep^{\tmop{gen}} (P<0,Z) & = & \chiep^{\tmop{gen}} (\ZZ (Q_{-},
  \R^{k+1})) - \chiep^{\tmop{gen}} (\ZZ ((Q,P), \R^{k})),\\
  2  \chiep^{\tmop{gen}}_{\mathfrak{S}_{k}} (P<0,Z) & = &
  \chiep^{\tmop{gen}}_{\mathfrak{S}_{k}} (\ZZ (Q_{-}, \R^{k+1})) -
  \chiep^{\tmop{gen}}_{\mathfrak{S}_{k}} (\ZZ ((Q,P), \R^{k})).
\end{eqnarray*}

\paragraph{\sc Complexity analysis}
The complexity of the algorithm is $(\omega  k d)^{O (D'')}$,
where $D'' = \sum_{i=1}^\omega \min(k_i,4d)$,
using the complexity analysis of Algorithm
\ref{alg:ep-general-BM}.

We are now ready to describe our algorithm for computing the
Euler-Poincar{\'e} characteristic of the realizations of sign conditions.

{\small
\begin{algorithm}[H]
\caption{(Generalized Euler-Poincar{\'e} Characteristic of Sign
Conditions)}
\label{alg:ep-sign-conditions}
\begin{algorithmic}[1]
\INPUT
\Statex{
\begin{enumerate}
    \item A tuple $\mathbf{k}= (k_{1}, \ldots,k_{\omega}) \in
    \Z_{>0}^{\omega}$, with $k= \sum_{i=1}^{\omega} k_{i}$.
    
    \item A set of $s$ polynomials $\mathcal{P}= \{ P_{1}, \ldots,P_{s} \}
    \subset \D [ \X^{(1)}, \ldots,\X^{(\omega)} ]$,
    where each $\X^{(i)}$ is a block of $k_{i}$ variables, and
    each polynomial in $\mathcal{P}$ is symmetric in each block of variables
    $\X^{(i)}$ and of degree at most $d$.
  \end{enumerate}
  }
  \OUTPUT
  \Statex{
   The lists
  $\chiep^{\gen} (\mathcal{P})$,
  $\chiep^{\gen}_{\mathfrak{S}_{\mathbf{k}}} (
  \mathcal{P})$.
  }
  
\PROCEDURE

    \State{$\mathcal{P} \gets \{ P_{1}, \ldots,P_{s} \}$, where 
    $\mathcal{P}_{i} = \{ P_{1}, \ldots,P_{i} \}$. Compute $\SIGN (
    \mathcal{P})$ using Algorithm \ref{alg:sampling} (Sampling).
    } \label{alg:ep-sign-conditions:step1}
    
    \State{Determine a list $A (\mathcal{P})$ adapted to sign
    determination for $\mathcal{P}$ on $Z$ using 
    Algorithm 10.83 (Adapted Matrix) in \cite{BPRbook3}.
    } \label{alg:ep-sign-conditions:step2}
    
    \State{ Define $A=A (\mathcal{P})$, $M=M (\mathcal{P}^{A}, \SIGN
    (\mathcal{P}))$.} \label{alg:ep-sign-conditions:step3}
    
    \State{Compute $\EQ (\mathcal{P}^{A})$,
    $\EQ_{\mathfrak{S}_{\mathbf{k}}} (\mathcal{P}^{A})$ using repeatedly
    Algorithm \ref{12:alg:speuler} (Euler-Poincar{\'e}-query).
    } \label{alg:ep-sign-conditions:step4}
    
    \State{ Using
    \begin{eqnarray*}
      M \cdot \chiep^{\tmop{gen}} (\mathcal{P,\mathbb{Q}}) & = & \EQ (
      \mathcal{P}^{A}),\\
      M \cdot \chiep^{\tmop{gen}}_{\mathfrak{S}_{\mathbf{k}}} (
      \mathcal{P,\mathbb{Q}}) & = & \EQ_{\mathfrak{S}_{\mathbf{k}}} (
      \mathcal{P}^{A}).
    \end{eqnarray*}
    and the fact that $M$ is invertible, compute $\chiep^{\tmop{gen}} (
    \mathcal{P}$, $\chiep^{\tmop{gen}}_{\mathfrak{S}_{\mathbf{k}}} (
    \mathcal{P})$.
    } \label{alg:ep-sign-conditions:step5}
 \end{algorithmic}
 \end{algorithm}
}

\paragraph{\sc Proof of correctness}
The correctness follows from the correctness of Algorithm \ref{alg:sampling}
and the proof of correctness of the corresponding algorithm 
(Algorithm 13.12) in \cite{BPRbook3}.

\paragraph{\sc Complexity analysis}
The complexity analysis is very similar to that of 
Algorithm 13.12 in \cite{BPRbook3}.
The only difference is the use of the bound on $\card
(\mathcal{P})$ afforded by Proposition \ref{prop:SIGN} in the symmetric
situation instead of the usual non-symmetric bound. By Proposition
\ref{prop:SIGN}

\[ \card (\SIGN (\mathcal{P})) \leq s^{D'}  d^{O (
   D'')}, 
\]
where $D'  = \sum_{i=1}^{\omega} \min (k_{i},d)$,
and $D''= \sum_{i=1}^{\omega} \min (k_{i},4d)$. The
number of calls to to Algorithm \ref{12:alg:speuler}
(Euler-Poincar{\'e}-query) is equal to~$\card   (\SIGN (\mathcal{P}
))$. The calls to Algorithm \ref{12:alg:speuler} (Euler-Poincar{\'e}-query)
are done for polynomials which are products of at most
\[ \log (\card (\SIGN (\mathcal{P}))) =O (D''
   \log  d+ D' \log  s)) 
\]
products of polynomials of the form $P$ or $P^{2}$, $P \in \mathcal{P}$ by
Proposition 10.84 in {\cite{BPRbook3}}, 
hence of degree bounded by $D=O (d
(D'' \log  d+ D' \log  s))$. Using the complexity analysis of
Algorithm \ref{alg:sampling} (Sampling) and the complexity analysis of
Algorithm \ref{12:alg:speuler} (Euler-Poincar{\'e}-query), the number of
arithmetic operations is bounded by 

\[ s^{D'} k^{d } d^{O (D'D'')} + s^{D'} d^{O(D'')} ( k \omega D)^{O (D''')}, \]
where 
$D=d  (D'' \log  d+ D'\log  s))$,
$D' = \sum_{i=1}^{\omega} \min (k_{i},d)$,
$D'' = \sum_{i=1}^{\omega} \min (k_{i},d)$, and
$D''' = \sum_{i=1}^\omega \min(k_i, 2D)$.

The algorithm also involves the inversion matrices of size $s^{D'} d^{O (D'')}$ with integer
coefficients.

{\small
\begin{algorithm}[H]
\caption{(Computing generalized Euler-Poincar{\'e}
characteristic of symmetric semi-algebraic sets)}
\label{alg:ep-sa}
\begin{algorithmic}[1]
\INPUT
\Statex{
\begin{enumerate}
    \item A tuple $\mathbf{k}= (k_{1}, \ldots,k_{\omega}) \in
    \Z_{>0}^{\omega}$, with $k= \sum_{i=1}^{\omega} k_{i}$.
    
    \item A set of $s$ polynomials $\mathcal{P}= \{ P_{1}, \ldots,P_{s} \}
    \subset \D [ \X^{(1)}, \ldots,\X^{(\omega)} ]$,
    where each $\X^{(i)}$ is a block of $k_{i}$ variables, and
    each polynomial in $\mathcal{P}$ is symmetric in each block of variables
    $\X^{(i)}$, and of degree at most $d$.
    
    \item A $\mathcal{P}$-semi-algebraic set $S$, described by
    \begin{eqnarray*}
      S & = & \bigcup_{\sigma \in \Sigma} \RR \left(\sigma, \R^{k} \right),
    \end{eqnarray*}
    where $\Sigma \subset \{ 0,1,-1 \}^{\mathcal{P}}$ of sign conditions on
    $\mathcal{P}$.
\end{enumerate}
}
\OUTPUT
\Statex{  $\chi^{\gen} (S
 )$ and $\chi^{\gen}_{\mathfrak{S}_{\mathbf{k}}} (
  S)$.
  }

\PROCEDURE
    \State{Compute using Algorithm \ref{alg:sampling} the set
    $\tmop{SIGN} (\mathcal{P})$.} \label{alg:ep-sa:step1}
    
    \State{Identify $\Gamma = \tmop{SIGN} (\mathcal{P}) \cap \Sigma$.} \label{alg:ep-sa:step2}
    
    \State{Compute using Algorithm \ref{alg:ep-general-BM},
    $\chi^{\tmop{gen}} (\mathcal{P})$,
    $\chi^{\tmop{gen}}_{\mathfrak{S}_{\mathbf{k}}} (\mathcal{P})$.} \label{alg:ep-sa:step3}
    
    \State{Compute using $\chi^{\tmop{gen}} (\mathcal{P}),
    \chi^{\tmop{gen}}_{\mathfrak{S}_{k}} (\mathcal{P})$, and $\Gamma,$
    \begin{eqnarray*}
      \chi^{\tmop{gen}} (S) & = & \sum_{\sigma \in \Sigma} \chi^{\tmop{gen}}
      (\sigma),\\
      \chi^{\tmop{gen}}_{\mathfrak{S}_{\mathbf{k}}} (S) & = & \sum_{\sigma
      \in \Sigma} \chi^{\tmop{gen}}_{\mathfrak{S}_{\mathbf{k}}} (\sigma).
    \end{eqnarray*}
    } \label{alg:ep-sa:step4}
  
\end{algorithmic}
\end{algorithm}
}

\paragraph{\sc Proof of correctness}
The correctness of Algorithm \ref{alg:ep-sa} follows from the correctness of
Algorithms \ref{alg:sampling} and \ref{alg:ep-sign-conditions}, and the
additive property of the generalized Euler-Poincar{\'e} characteristic (see
Definition \ref{def:ep-general}). 

\paragraph{\sc Complexity analysis}
The complexity is dominated by Step \ref{alg:ep-sa:step3}, and is thus bounded by
\[ \card (
\Sigma)^{O (1)} + s^{D'} k^{d } d^{O (D'D'')} + s^{D'} d^{O(D'')} ( k \omega D)^{O (D''')}, \]
where 
$D=d  (D'' \log  d+ D'\log  s))$,
$D' = \sum_{i=1}^{\omega} \min (k_{i},d)$,
$D'' = \sum_{i=1}^{\omega} \min (k_{i},d)$, and
$D''' = \sum_{i=1}^\omega \min(k_i, 2D)$.

The algorithm also involves the inversion matrices of size $s^{D'} d^{O (D'')}$ with integer
coefficients.

\begin{proof}[Proof of Theorem \ref{thm:algorithm-sa}] The
correctness and the complexity analysis of Algorithm \ref{alg:ep-sa}
prove Theorem \ref{thm:algorithm-sa}.
\end{proof}

\section{Appendix}
\label{sec:appendix}

\begin{notation}
\label{not:pi-of-x}
For $\x \in \R^k$ or $\C^k$, let $G_\x$ be the isotropy subgroup of $\x$ with respect to the action of
$\mathfrak{S}_k$ on $\R^k$ or $\C^k$ permuting coordinates. Then, it is easy to verify that 
\[
G_\x \cong \mathfrak{S}_{\ell_1} \times \cdots \times \mathfrak{S}_{\ell_m},
\]
where $k \geq \ell_1 \geq \ell_2 \geq \cdots \geq \ell_m > 0, \sum_i \ell_i = k$,
and $\ell_1,\ldots,\ell_m$  are the  
cardinalities of the sets 
\[
\{i \mid 1 \leq i \leq k, x_i = x\}, x \in \bigcup_{i=1}^k \{x_i \}
\] 
in non-decreasing order.
We denote by $\pi(\x)$ the partition $(\ell_1,\ldots,\ell_m) \in \Pi_k$.

More generally, for $\mathbf{k}= (k_{1}, \ldots,k_{\omega})
  \in \Z_{>0}^{\omega}$, with $k= \sum_{i=1}^{\omega} k_{i}$, and 
  $\x = (\x^{(1)},\ldots,\x^{(\omega)}) \in \R^k$, where each $\x^{(i)} \in \R^{k_i}$, we denote 
  \[
  \boldpi(\x) = (\pi(\x^{(1)}),\ldots,\pi(\x^{(\omega)})) \in \boldPi_\kk.
  \]
\end{notation}

In the following proposition we use Notation \ref{not:L-fixed}.
\begin{proposition}\cite[Proposition 7]{BC2013}
  \label{prop:orthogonal} Let $L'_{\fixed}
  \subset L_{\fixed}$ any subspace of
  $L_{\fixed}$, and $I  \subset \{ (i,j) \mid 1
  \leq i \leq \omega,1 \leq j \leq \ell_{i} \}$. Then the following hold.
  \begin{enumerate}[A.]
    \item
    \label{itemlabel:prop:orthogonal:1}
     The dimension of $L_{\fixed}$ is equal to $ 
    \sum_{i=1}^{\omega} \ell_{i} -1= \length (\boldpi) -1$.
    \item 
    \label{itemlabel:prop:orthogonal:2}
    The product over $i \in [ 1, \omega ]$ of the subgroups
    $\mathfrak{S}_{\pi^{(i)}_{1}} \times \mathfrak{S}_{\pi^{(i)}_{2}}
    \times \cdots \times \mathfrak{S}_{\pi^{(i)}_{\ell_{i}}}$ acts trivially
    on $L_{\fixed}$.
    
    \item
    \label{itemlabel:prop:orthogonal:3}
     For each $i,j,1 \leq i \leq \omega,1 \leq j \leq \ell_{i}$, $M^{(i
   )}_{j}$ is an irreducible representation of $\mathfrak{S}_{\pi^{(i
   )}_{j}}$, and the action of $\mathfrak{S}_{\pi^{(i')}_{j'}}$ on $M^{(
    i)}_{j}$ is trivial if $(i,j) \neq (i',j')$.
    
    \item 
    \label{itemlabel:prop:orthogonal:4}
    There is a direct decomposition $L=L_{\fixed} \oplus \left(
    \bigoplus_{1 \leq i \leq \omega,1 \leq j \leq \ell_{i}} M^{(i)}_{j}
    \right)$.
    
    \item 
    \label{itemlabel:prop:orthogonal:5}
    Let $\mathbf{D}$ denote the unit disc in the subspace
    $L_{\fixed}' \oplus \left(\bigoplus_{(i,j) \in I} M^{(i)}_{j}
    \right)$. Then, the space of orbits of the pair $(\mathbf{D}, \partial
    \mathbf{D})$ under the action of $\mathfrak{S}_{\mathbf{k}}$ is homotopy
    equivalent to $(\ast, \ast)$ if $I \neq \emptyset$. Otherwise, the
    space of orbits of the pair $(\mathbf{D}, \partial \mathbf{D})$ under
    the action of $\mathfrak{S}_{\mathbf{k}}$ is homeomorphic to $(
    \mathbf{D}, \partial \mathbf{D})$.
  \end{enumerate}
\end{proposition}

\begin{lemma}\cite[Lemma 5]{BC2013}
  \label{lem:equivariant_morseA} Then, for $1 \leq i<N$, and for each $c \in
  [ c_{i},c_{i+1})$, $\phi_{\mathbf{k}} (S_{\leq c})$ is
  semi-algebraically homotopy equivalent to $\phi_{\mathbf{k}} (S_{\leq
  c_{i}})$.
\end{lemma}

Let $L^{+} (\x) \subset
  L$ and $L^{-} (\x) \subset L$ denote the positive and negative eigenspaces
  of the Hessian of the function $p_{1}^{(k)}$ restricted to $W$ at $\x$.
  Let $\ind^{-} (\x) = \dim  L^{-} (\x)$.

  The proof of the following lemma follows closely the proof of a similar result  (Lemma 6) in \cite{BC2013}.
  
\begin{lemma}
  \label{lem:equivariant_morseB}
  Let $G_c$ denote a set of representatives of orbits of critical points $\x$  of  $F$ restricted to $W$ with $F(\x) = c$.
    Then,  for all small enough $t>0$,
    \begin{eqnarray}
    \label{eqn:inequality2}
      \chi^{\mathrm{top}}(\phi_{\mathbf{k}} (S_{\leq c}),\F) & =& \chi^{\mathrm{top}}(
      \phi_{\mathbf{k}} (S_{\leq c-t}),\F) +
      \sum_{\x}
      (-1)^{\ind^-(\x)},
      \end{eqnarray}
      where the sum is taken over all $\x \in G_c$ with $\sum_{1 \leq i \leq k} \dfrac{\partial P}{\partial X_{i}} (\x)  <0$.
\end{lemma}

\begin{proof} We first prove the proposition for $\R =
\mathbb{R}$. 
We will also assume that the function $F$ takes distinct values on the distinct orbits
  of the critical points of $F$ restricted to $W$ for ease of exposition of the proof.  Since
  the topological changes at the critical values are local near the critical points which are assumed to be isolated, the general case follows easily using a standard partition of unity argument. 
  Also, note that the value of $\ind^{-1}(\x)$ (respectively,   $\sign(\sum_{1 \leq i \leq k} \dfrac{\partial P}{\partial X_{i}} (\x) )$)
 are equal for all critical points $\x$ belonging to one orbit.
 
    Suppose that for each critical point $\x \in W$, with $F (\x) =c$, 
    \[
    \sum_{1 \leq i \leq k} \dfrac{\partial P}{\partial X_{i}} (\x)  >0.
    \]
    We prove that in this case,
    for for all small enough $t>0$,
    \begin{eqnarray}
    \label{eqn:inequality1}
      \chi^{\mathrm{top}}(\phi_{\mathbf{k}} (S_{\leq c}),\F) & = & \chi^{\mathrm{top}}(
      \phi_{\mathbf{k}} (S_{\leq c-t}),\F).
    \end{eqnarray}
    
 If 
 \[
 \sum_{1 \leq i \leq k} \dfrac{\partial P}{\partial X_{i}} (\x) 
  >0,
  \]
  then $S_{\leq c}$ retracts
  $\mathfrak{S}_{\mathbf{k}}$-equivariantly to a space $S_{\leq c-t} \cup_{B}
  A$ where the pair $(A,B) = \coprod_{\x} ( A_{\x},B_{\x})$,  and where the
  disjoint union is taken over the set critical points $\x$ with $F (\x) =c$,
  and each pair $(A_{\x},B_{\x})$ is homeomorphic to the pair $(
  \mathbf{D}^{i} \times [ 0,1 ], \partial \mathbf{D}^{i} \times [ 0,1 ] \cup
  \mathbf{D}^{i} \times \{ 1 \})$,
  where $i$ is the dimension of the negative eigenspace of the Hessian  of the
  function $e_1^{(k)}$ restricted to $W$ at $\x$.
  This follows from the basic Morse theory
  (see \cite[Proposition 7.21]{BPRbook3}). Since the pair $(
  \mathbf{D}^{i} \times [ 0,1 ], \partial \mathbf{D}^{i} \times [ 0,1 ] \cup
  \mathbf{D}^{i} \times \{ 1 \})$ is homotopy equivalent to $(\ast, \ast
 )$, $S_{\leq c}$ is homotopy equivalent to $S_{\leq c-t}$, and it follows
  that $\phi_{\mathbf{k}} (S_{\leq c})$ is homotopy equivalent to
  $\phi_{\mathbf{k}} (S_{\leq c-t})$ as well, because of the fact that
  retraction of $S_{\leq c}$ to $S_{\leq c-t} \cup_{B} A$ is chosen to be
  equivariant. The equality \eqref{eqn:inequality1} then follows immediately. \\

  We now consider the case when 
  for each critical point $\x \in W$, with $F (\x) =c$,
  $\sum_{1 \leq i \leq k} \dfrac{\partial
  P}{\partial X_{i}} (\x)  < 0$. Let $T_{\x} W$ be the tangent space of $W$ at
  $\x$. The translation of $T_{\x} W$ to the origin is then the linear subspace
  $L \subset \R^{k}$ defined by $\sum_{i} X_{i} =0$. 
  Let $\x  \in
  L_{\boldpi}$ where $\boldpi= (\pi^{(1)}, \ldots, \pi^{(
  \omega)}) \in \boldPi_{\mathbf{k}}$, where for each $i,1 \leq i
  \leq \omega$, $\pi^{(i)} = (\pi^{(i)}_{1}, \ldots, \pi^{(i
 )}_{\ell_{i}}) \in \Pi_{k_{i}}$. The subspaces $L^{+} (\x),L^{-} (\x
 )$ are stable under the the natural action of the subgroup $\prod_{1 \leq i
  \leq \omega,1 \leq j \leq \ell_{i}} \mathfrak{S}_{\pi^{(i)}_{j}}$ of
  $\mathfrak{S}_{\mathbf{k}}$. For $1 \leq i \leq \omega,1 \leq j \leq
  \ell_{i}$, let $L^{(i)}_{j}$ denote the subspace $L \cap L_{\pi^{(i
 )}_{j}}$ of $L$, and $M^{(i)}_{j}$ the orthogonal complement of $L^{(i
 )}_{j}$ in $L$. Let $L_{\fixed} =L \cap L_{\boldpi}$. It
  follows from 
  Parts \eqref{itemlabel:prop:orthogonal:2}, \eqref{itemlabel:prop:orthogonal:3}, and\eqref{itemlabel:prop:orthogonal:4}
  of Proposition \ref{prop:orthogonal} that:
  \begin{enumerate}[i]
    \item For each $i,j, \,1 \leq i \leq \omega,1 \leq j \leq
    \ell_{i}$, $M^{(i)}_{j}$ is an irreducible representation of
    $\mathfrak{S}_{\pi_{i}}$, and the action of $\mathfrak{S}_{\pi^{(i'
   )}_{j'}}$ on $M^{(i)}_{j}$ is trivial if $(i,j) \neq (i',j')$.
    Hence, for each $i,   j  ,1 \leq i \leq \omega,1 \leq j
    \leq \ell_{i}$, $L^{-} (p) \cap M^{(i)}_{j} =0  \tmop{or}  M^{(i
   )}_{j}$.
    
    \item The subgroup $\prod_{1 \leq i \leq \omega,1 \leq j \leq \ell_{i}}
    \mathfrak{S}_{\pi^{(i)}_{j}}$ of $\mathfrak{S}_{\mathbf{k}}$ acts
    trivially on $L_{\fixed}$.
    
    \item There is an orthogonal decomposition $L=L_{\fixed} \oplus
    \left(\bigoplus_{1 \leq i \leq \omega,1 \leq j \leq \ell_{i}} M^{(i
   )}_{j} \right)$.
  \end{enumerate}
  It follows that
  \begin{eqnarray*}
    L^{-} (p) & = & L'_{\fixed}   \oplus \left(\bigoplus_{(i,j) \in
    I} M^{(i)}_{j} \right),
  \end{eqnarray*}
  where $L'_{\fixed}$ is some subspace of $L_{\fixed}$ and $I
  \subset \{ (i,j) \mid 1 \leq i \leq \omega,1 \leq j \leq \ell_{i} \}$.
  
  It follows from the proof of Proposition 7.21 in \cite{BPRbook3} that
  for all sufficiently small $t >0$ then $S_{\leq c}$ is retracts
  $\mathfrak{S}_{\mathbf{k}}$-equivariantly to a space $S_{\leq c-t} \cup_{B}
  A$ where the pair $(A,B) = \coprod_{\x} ( A_{\x},B_{\x})$, and the
  disjoint union is taken over the set critical points $\x$ with $F (\x) =c$,
  and each pair $(A_{\x},B_{\x})$ is homeomorphic to the pair $(\mathbf{D}^{\ind^{-} (\x)}, \partial \mathbf{D}^{\ind^{-} (\x
 )})$. It follows from the fact that the retraction mentioned above is
  equivariant that $\phi_{\mathbf{k}} (S_{\leq c})$ retracts to a space
  obtained from $\phi_{\mathbf{k}} (S_{\leq c-t})$ by glueing
  $\orbit_{\mathfrak{S}_{\mathbf{k}}} \left(\coprod_{\x} A_{\x} \right)$
  along $\orbit_{\mathfrak{S}_{\mathbf{k}}} \left(\coprod_{\x} B_{\x}
  \right)$. Now there are the following cases to consider:
  \begin{enumerate}[(a)]
    \item $\ind^{-} (\x) =0$. In this case
    \[
    \orbit_{\mathfrak{S}_{\mathbf{k}}} (\coprod_{\x} A_{\x},
    \coprod_{\x} B_{\x})
    \] 
    is homotopy equivalent to $(\ast, \emptyset)$.
    
    \item $L^{-} (\x)   \subset L_{\fixed}$ (i.e. $I= \emptyset$ in
    this case). In this case 
    \[
    \orbit_{\mathfrak{S}_{\mathbf{k}}}
    (\coprod_{\x} A_{\x}, \coprod_{\x} B_{\x})
    \] 
    is homeomorphic to
    $(\mathbf{D}^{\ind^{-} (\x)}, \partial \mathbf{D}^{\ind^{-}
    (\x)})$ by 
    Part \eqref{itemlabel:prop:orthogonal:5}
    of Proposition \ref{prop:orthogonal}.
    
    \item Otherwise, there is a non-trivial action on $L^{-} (\x)$ 
    of  the group 
    \[
    \prod_{(i,j) \in I} \mathfrak{S}_{\pi^{(i)}_{j}},
    \] 
    and it
    follows from 
    Part  \eqref{itemlabel:prop:orthogonal:5}
    of Proposition \ref{prop:orthogonal} that in this
    case 
    \[
    \orbit_{\mathfrak{S}_{\mathbf{k}}}(\coprod_{\x} A_{\x}, \coprod_{\x} B_{\x})
    \]
     is homotopy equivalent to $(\ast, \ast)$.
  \end{enumerate}
  The equality \eqref{eqn:inequality2} 
  follow immediately from \eqref{eqn:MV2-ep}.
  
This finishes the proof in case $\R =\mathbb{R}$. The statement over a
general real closed field $\R$ now follows by a standard application of the
Tarski-Seidenberg transfer principle (see for example the proof of Theorem
7.25 in \cite{BPRbook3}).\end{proof}

\begin{proof}[Proof of Theorem \ref{thm:equivariant-ep}]
The proof follows immediately from Lemmas  \ref{lem:equivariant_morseA} and \ref{lem:equivariant_morseB}.
\end{proof}

\bibliographystyle{abbrv}
\bibliography{master}
\end{document}